\date{\today}
\documentclass[12pt]{amsart}
\usepackage[margin=1.25in, headheight=14pt]{geometry}

\usepackage{amsmath}
\usepackage{amsfonts}
\usepackage{amssymb}
\usepackage{amsthm}
\usepackage{amsopn}
\usepackage{tikz}
\usetikzlibrary{calc}
\usetikzlibrary{patterns}
\usetikzlibrary{matrix}

\theoremstyle{definition}
\newtheorem{theorem}{Theorem}[section]
\newtheorem{definition}[theorem]{Definition}
\newtheorem{lemma}[theorem]{Lemma}
\newtheorem{corollary}[theorem]{Corollary}

\newtheorem{example}[theorem]{Example}
\newtheorem{proposition}[theorem]{Proposition}
\newtheorem{remark}[theorem]{Remark}
\newtheorem{question}[theorem]{Question}
\newtheorem{conjecture}[theorem]{Conjecture}

\newcommand{\pdim}{\text{pd}}
\newcommand{\reg}{\text{reg}}

\DeclareMathOperator{\kk}{\Bbbk}

\newcommand{\betti}[2]{\beta_{#1,#2}}

\newcommand{\homol}[1]{\widetilde{H}_{#1}}
\newcommand{\dhomol}[1]{\dim\widetilde{H}_{#1}}
\newcommand{\Gc}{\widehat{G^c}}

\newcommand{\Ind}[1]{\text{Ind(#1)}}
\newcommand{\Susp}{\text{Susp}}
\newcommand{\Cone}{\text{Cone}}
\newcommand{\sdd}{\text{sd}_4}
\newcommand{\dotcup}{\ensuremath{\mathaccent\cdot\cup}} 

\begin{document}


\title{Jump Sequences of Edge Ideals}
\author[G. Whieldon]{Gwyn Whieldon}
\address[Gwyn Whieldon]{Cornell University\\
Ithaca, NY  14850}
\thanks{Wishes to thank her advisor M. Stillman and A. Hoefel, E. Nevo, and I. Peeva for many productive discussions.}
\email{whieldon@math.cornell.edu}

\begin{abstract}
Given an edge ideal of graph G, we show that if the first nonlinear strand in the resolution of $I_G$ is zero until homological stage $a_1$, then the next nonlinear strand in the resolution is zero until homological stage $2a_1$.   Additionally, we define a sequence, called a \emph{jump sequence}, characterizing the highest degrees of the free resolution of the edge ideal of G via the lower edge of the Betti diagrams of $I_G$.  These sequences strongly characterize topological properties of the underlying Stanley-Reisner complexes of edge ideals, and provide general conditions on construction of clique complexes on a fix set of vertices.  We also provide an algorithm for obtaining a large class of realizable jump sequences and classes of Gorenstein edge ideals achieving high regularity.
\end{abstract}
\maketitle



\section{Introduction}\label{intro}
Let $G$ be a simple graph (e.g. no loops or multiple edges) on $n$ vertices, and $e$ edges, denoted $G=(V,E)$ with $V$ the vertex set of $G$ and $E$ the edge set.  Let $\kk$ be a field of characteristic zero and $R=\kk[x_1,..,x_n]$ the polynomial ring over $\kk$ with a generator for each vertex of $G$.
\begin{definition}  Let $G$, $R$ as above.  Then the \emph{edge ideal of $G$}, denoted $I_G$ is the squarefree monomial ideal given by
$$I_G=(x_ix_j: \{i,j\}\in E).$$
\end{definition}
Ideals of this form have been of great interest recently, with excellent surveys here \cite{HaVanTu07} and here \cite{HaVanTu08}.  After their introduction by Villareal \cite{Vil90}, they have been studied extensively, with the goal of building a dictionary between graph properties of G and algebraic properties of $I_G$.\\
\\
Two invariants of $I_G$ of particular interest are the regularity and projective dimension, which respectively measure the width and length of the resolution.
\begin{definition}  We say that the \emph{regularity of $I_G$}, or $\reg(I_G)$ is
$$\reg(I_G)=\max\{j-i+1: \betti{i}{j}(I_G)\neq 0\}.$$
We say that the \emph{projective dimension of $I_G$}, or $\pdim(I_G)$, is
$$\pdim(I_G)=\max\{i:\betti{i}{j}(I_G)\neq 0\}.$$
\end{definition}
A property of $G$ which provides an immediate lower bound for the regularity of $I_G$ is the \emph{induced matching number of $G$.}  It has also been shown that the \emph{matching number} of $G$ provides an upper bound for the regularity of $I_G$.  Some recent work on on other upper bounds for regularity has can be found here \cite{Woo10}.

\begin{definition}  We say that $G$ has induced matching number $k$, or
$$\Ind{G}=k$$
if the largest subset of edges that can be chosen to be completely disconnected in the induced subgraph of $G$ restricted to those vertices is of size $k$.  We say that $G$ has matching number $k$, or
$$M(G)=k$$
if the largest mutually disjoint set of edges is of size k.
\end{definition}

\begin{example}  Considering the cycle graphs of lengths 4, 5, and 6, 




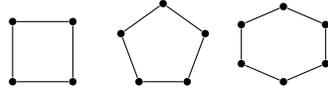
\begin{center}
\begin{tikzpicture}
[scale=.8,auto=left,vertices/.style={circle, fill=black, inner sep=1pt}]
\node [vertices] (a) at (0,0){};
\node [vertices] (b) at (1,0){};
\node [vertices] (d) at (0,1){};
\node [vertices] (c) at (1,1){};
\foreach \from/\to in {a/b,b/c,c/d,a/d}
	\draw [-] (\from) -- (\to);

\node [vertices] (e) at (0+.1+2,0){};
\node [vertices] (f) at (1-.1+2,0){};
\node [vertices] (i) at (-.2+2,1-.2){};
\node [vertices] (g) at (1.2+2,1-.2){};
\node [vertices] (h) at (.5+2,1.3){};
\foreach \from/\to in {e/f,f/g,g/h,h/i,e/i}
	\draw [-] (\from) -- (\to);

\node [vertices] (j) at (.5+2+2,0){};
\node [vertices] (o) at (0-.2+2+2,0+.3){};
\node [vertices] (k) at (1+.2+2+2,0+.3){};
\node [vertices] (n) at (-.2+2+2,1-.05){};
\node [vertices] (l) at (1.2+2+2,1-.05){};
\node [vertices] (m) at (.5+2+2,1.3-.05){};
\foreach \from/\to in {j/k,k/l,l/m,m/n,n/o,o/j}
	\draw [-] (\from) -- (\to);
\end{tikzpicture}
\end{center}
denoted $C_4$, $C_5$ and $C_6$ respectively, we see that  $\Ind{$C_4$}=1$, $\Ind{$C_5$}=1$, $\Ind{$C_6$}=2$ and $M(C_4)=2$, $M(C_5)=2$, and $M(C_6)=3$.\\



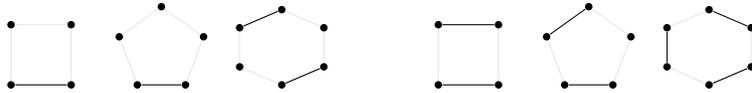
\begin{center}
\begin{tikzpicture}
[scale=.8,auto=left,vertices/.style={circle, fill=black, inner sep=1pt}]
\node [vertices] (a) at (0,0){};
\node [vertices] (b) at (1,0){};
\node [vertices] (d) at (0,1){};
\node [vertices] (c) at (1,1){};
\foreach \from/\to in {b/c,c/d,a/d}
	\draw [black!10,-] (\from) -- (\to);
\foreach \from/\to in{a/b}
	\draw [-] (\from) -- (\to);

\node [vertices] (e) at (0+.1+2,0){};
\node [vertices] (f) at (1-.1+2,0){};
\node [vertices] (i) at (-.2+2,1-.2){};
\node [vertices] (g) at (1.2+2,1-.2){};
\node [vertices] (h) at (.5+2,1.3){};
\foreach \from/\to in {f/g,g/h,h/i,e/i}
	\draw [black!10,-] (\from) -- (\to);
\foreach \from/\to in{e/f}
	\draw [-] (\from) -- (\to);

\node [vertices] (j) at (.5+2+2,0){};
\node [vertices] (o) at (0-.2+2+2,0+.3){};
\node [vertices] (k) at (1+.2+2+2,0+.3){};
\node [vertices] (n) at (-.2+2+2,1-.05){};
\node [vertices] (l) at (1.2+2+2,1-.05){};
\node [vertices] (m) at (.5+2+2,1.3-.05){};
\foreach \from/\to in {j/k,k/l,l/m,m/n,n/o,o/j}
	\draw [black!10,-] (\from) -- (\to);
\foreach \from/\to in{j/k,m/n}
	\draw [-] (\from) -- (\to);

\end{tikzpicture}
\;\;\;\;\;\;\;\;\;\;
\begin{tikzpicture}
[scale=.8,auto=left,vertices/.style={circle, fill=black, inner sep=1pt}]
\node [vertices] (a) at (0,0){};
\node [vertices] (b) at (1,0){};
\node [vertices] (d) at (0,1){};
\node [vertices] (c) at (1,1){};
\foreach \from/\to in {b/c,a/d}
	\draw [black!10,-] (\from) -- (\to);
\foreach \from/\to in{a/b,c/d}
	\draw [-] (\from) -- (\to);

\node [vertices] (e) at (0+.1+2,0){};
\node [vertices] (f) at (1-.1+2,0){};
\node [vertices] (i) at (-.2+2,1-.2){};
\node [vertices] (g) at (1.2+2,1-.2){};
\node [vertices] (h) at (.5+2,1.3){};
\foreach \from/\to in {f/g,g/h,e/i}
	\draw [black!10,-] (\from) -- (\to);
\foreach \from/\to in{e/f,h/i}
	\draw [-] (\from) -- (\to);

\node [vertices] (j) at (.5+2+2,0){};
\node [vertices] (o) at (0-.2+2+2,0+.3){};
\node [vertices] (k) at (1+.2+2+2,0+.3){};
\node [vertices] (n) at (-.2+2+2,1-.05){};
\node [vertices] (l) at (1.2+2+2,1-.05){};
\node [vertices] (m) at (.5+2+2,1.3-.05){};
\foreach \from/\to in {j/k,k/l,m/n,o/j}
	\draw [black!10,-] (\from) -- (\to);
\foreach \from/\to in{j/k,n/o,l/m}
	\draw [-] (\from) -- (\to);
\end{tikzpicture}
\end{center}
For these graphs, the regularity of $I_G$ is respectively $\reg(I_{C_4})=2$, $\reg(I_{C_5})=3$, and $\reg(I_{C_6})=3$.
\end{example}
In the case $G$ a tree, this number is precisely related to the regularity via the formula $\reg(I_G)=\Ind{G}+1$, noted in \cite{Zhe04}.  As seen in the example above for the regularity of the edge ideal of the 5-cycle, this fails for general graphs.
\begin{question}\label{regbounds}[Open] For graphs G with $\Ind{G}=1$\footnote{Equivalently, $G^c$ induced 4-cycle free, or $\betti{2}{4}(I_G)=0$.} is there a bound on the regularity of $I_G$?
\end{question}
Partial answers exist:  For graphs with $\Ind{G}=1$ with G claw-free the regularity of $I_G$ is at most 4.  In addition, for graphs of this form, we have that $I_G^2$ has a linear resolution \cite{Nev09}.  Other bounds for the regularity of $I_G$ have been provided in terms of the size of co-chordal covers \cite{Woo10}.  Examples of graphs which have $\Ind{G}=1$ and regularity of $I_G$ up to 5 are known, and we provide new general classes of graphs with $\Ind{G}=1$ and regularity $I_G=4$.  We also provide an example of graphs with $\Ind{G}=k$ and regularity as high as $\Ind{G}=4k+1$.\\
\\
We refine this question:
\begin{question}\label{mainquestion1}  Given any strictly increasing sequence of positive integers of the form $[k;a_1,...,a_{k-1}]$, does there exist a graph $G$ such that $\reg(I_G)=k+1$ and the first degree (i+k+1) syzygies of $I_G$ occurring at the $(a_i+1)$st homological stage of the resolution?
\end{question}
These $a_i$ are the sequence of numbers appearing below on the following Betti diagram, where the Betti diagram has been denoted in the style of Macaulay 2 via shifting the degree down by one in each adjacent row of the resolution.  In this paper, we classify some possible shapes of this lower edge of the resolution and demonstrate sequences $[k;a_1,...,a_{k-1}]$ which are prohibited from occurring in the resolution of an edge ideal $I_G$.
\begin{center}
\begin{tikzpicture} [scale=.8,auto=left,betti/.style={
execute at begin cell=\node\bgroup,
execute at end cell=\egroup;,
execute at empty cell=\node{$\cdot$};
},
vertices/.style={circle, fill=black, inner sep=0pt}]

\node [vertices] (r1) at (-7,.2){};
\node [vertices] (r2) at (-5.1+1.4,.2){};
\node [vertices] (r3) at (-5.1+1.4,-.7){};
\node [vertices] (r4) at (-1.6+1.4,-.7){};
\node [vertices] (r5) at (-1.6+1.4,-1.5){};
\node [vertices] (rk) at (1+1.4,-1.5){};
\node [vertices] (r6) at (1.7+1.4,-1.5){};
\node [vertices] (r7) at (1.7+1.4,-1.9){};
\node [vertices] (r8) at (2.2+1.4,-1.9){};
\node [vertices] (r9) at (2.2+1.4,-2.3){};
\node [vertices] (rj) at (3+1.4,-2.3){};
\node [vertices] (r10) at (5.6+1.4,-2.3){};
\node [vertices] (r11) at (5.6+1.4,-3.3){};
\node [vertices] (r12) at (5.9+1.4,-3.3){};

\foreach \from/\to in {r1/r2,r2/r3,r3/r4,r4/r5,r5/rk,rj/r10,r10/r11,r11/r12}
	\draw[-] (\from)--(\to);
\foreach \from/\to in {rk/r6,r6/r7,r7/r8,r8/r9,r9/rj}
	\draw[dotted] (\from)--(\to);

\node [vertices] (s1) at (-8.4+1.4,0){};
\node [vertices] (s2) at (-5.1+1.4,0){};
\draw[dashed] (s1)--(s2) node[midway, below] {$a_1$};

\node [vertices] (s3) at (-8.4+1.4,-.9){};
\node [vertices] (s4) at (-1.6+1.4,-.9){};
\draw[dashed] (s3)--(s4) node[midway, below] {$a_2$};

\node [vertices] (s5) at (-8.4+1.4,-1.7){};
\node [vertices] (s6) at (1.7+1.4,-1.7){};
\draw[dashed] (s5)--(s6);

\node [vertices] (s7) at (-8.4+1.4,-2.1){};
\node [vertices] (s8) at (2.2+1.4,-2.1){};
\draw[dashed] (s7)--(s8);

\node [vertices] (s3) at (-8.4+1.4,-2.5){};
\node [vertices] (s4) at (5.6+1.4,-2.5){};
\draw[dashed] (s3)--(s4) node[midway, below] {$a_{k-1}$};

\matrix [betti] {
-&0&1&2&3&$a_1+1$&$\cdots$&$a_2+1$&$\cdots$&$a_{k-1}-1$&$a_{k-1}$&$a_{k-1}+1$\\
total: &1&$\beta_{1}$&$\beta_{2}$&$\cdots$&$\beta_{a_1+1}$&$\cdots$&$\beta_{a_2+1}$&$\cdots$&$\beta_{a_{k-1}-1}$&$\beta_{a_{k-1}}$&$\cdots$\\
0: &1&&&&&&&&&\\
1:& &$\betti{1}{2}$&$\betti{2}{3}$&$\ast$&$\ast$&$\ast$&$\ast$&$\ast$\\
2:&\;&\;&\;&\;&$\betti{a_1+1}{a_1+3}$&$\ast$&$\ast$&$\ast$&$\ast$\\
3:&\;&\;&\;&\;&\;&\;&$\betti{a_2+1}{a_2+4}$&$\ast$&$\ast$&$\ast$&$\ast$\\
$\vdots$&\;&\;&\;&\;&\;&\;&\;&\;&$\ast$&$\ast$\\
k:& \;&\;&\;&\;&\;&\;&\;&\;&\;&\;&$\betti{a_{k-1}+1}{s}$\\
};
\end{tikzpicture}
\end{center}
This question is a strengthening of Question \ref{regbounds}, as the former question can be rephrased in terms of this sequence.
\begin{question}[Equivalent to Question \ref{regbounds}]  (Open) Is there a bound on the length of a sequence $[k;a_1,a_2,...,a_{k-1}]$ of the form in Question \ref{mainquestion1} for an ideal $I_G$ if $a_1\geq 2$?
\end{question}
We answer Question \ref{mainquestion1} negatively, although sharp conditions for a given sequence to have a corresponding edge ideal remain elusive.  One necessary condition for such a sequence to arise from a graph G:  For all edge ideals $I_G$, if the first nonlinear betti numbers occur at homological stage i of the resolution, the next nonlinear betti numbers must occur at stage 2i or later.  To rephrase this in the language of these sequences:
\begin{theorem}\label{maintheoremintro}  Given a sequence $[k;a_1,a_2,...,a_{k-1}]$ of the form in Question \ref{mainquestion1},
$$2a_1\leq a_2.$$
\end{theorem}
\begin{example}  This theorem prohibits Betti diagrams of any edge ideals $I_G$ from having the following shapes:
\begin{center}
\begin{tikzpicture} [betti/.style={
execute at begin cell=\node\bgroup,
execute at end cell=\egroup;,
execute at empty cell=\node{$\cdot$};
},vertices/.style={circle, fill=black, inner sep=0pt}] 
\matrix [betti] {
-&0&1&2&3&4&5&6&$\cdots$\\
total: &1&$\ast$&$\ast$&$\ast$&$\ast$&$\ast$&$\ast$&$\cdots$\\
0: &1&&&&&&&\\
1:&&$\ast$&$\ast$&$\circ$&$\circ$&$\circ$&$\circ$&$\cdots$\\
2:&&&&$\ast$&$\ast$&$\circ$&$\circ$&$\cdots$\\
3:&&&&&$\ast$&$\ast$&$\circ$&$\cdots$\\
};
\node [vertices] (1) at (-1,0){};
\node [vertices] (2) at (-1,-.6){};
\node [vertices] (3) at (-0.05,-.6){};
\node [vertices] (4) at (-0.05,-1.1){};
\node [vertices] (5) at (.45,-1.1){};
\node [vertices] (6) at (.45,-1.6){};
\node [vertices] (7) at (.7,-1.6){};

\foreach \from/\to in {1/2,2/3,3/4,4/5,5/6,6/7}
	\draw [-] (\from)--(\to);

\end{tikzpicture}\;\;\;\;\;
\begin{tikzpicture} [betti/.style={
execute at begin cell=\node\bgroup,
execute at end cell=\egroup;,
execute at empty cell=\node{$\cdot$};
},vertices/.style={circle, fill=black, inner sep=0pt}] 
\matrix [betti] {
-&0&1&2&3&4&5&6&$\cdots$\\
total: &1&$\ast$&$\ast$&$\ast$&$\ast$&$\ast$&$\ast$&$\cdots$\\
0: &1&&&&&&&\\
1:&&$\ast$&$\ast$&$\ast$&$\circ$&$\circ$&$\circ$&$\cdots$\\
2:&&&&&$\ast$&$\ast$&$\circ$&$\cdots$\\
3:&&&&&&&$\circ$&$\cdots$\\
};
\node [vertices] (1) at (-1,0){};
\node [vertices] (2) at (-1,-.6){};
\node [vertices] (3) at (.46,-.6){};
\node [vertices] (4) at (.46,-1.1){};
\node [vertices] (5) at (1.45,-1.1){};
\node [vertices] (6) at (1.45,-1.6){};
\node [vertices] (7) at (1.65,-1.6){};

\foreach \from/\to in {1/2,2/3,3/4,4/5,5/6,6/7}
	\draw [-] (\from)--(\to);
\end{tikzpicture}
\end{center}
\end{example}
For low $a_1$, this lower bound on $a_2$ is not sharp.  Theorem \ref{maintheoremintro} gives that if $a_1=2$, the left diagram above is impossible, i.e. $a_2\geq 4$.  Similarly, we have that if $a_1=2$, then $a_2\geq 6$.  However, slightly stronger lower bounds on $a_2$ bounds hold (and we conjecture a sharp lower bound on $a_2$ in terms of $a_1$):
\begin{proposition}\cite{Wh11}  Let $G$ be a simple graph, $I_G$ its edge ideal, and $[k;a_1,...,a_{k-1}]$ be the sequence in \ref{mainquestion1}.  Then the following hold:
\begin{enumerate}
\item  If $a_1=2$, then $a_2\geq 6$.
\item If $a_1=3$, then $a_2\geq 9$.
\end{enumerate}
\end{proposition}
The proofs of these are technical, and we only include the proof that if $a_1=2$ then $a_2\geq 5$ here.  We also conjecture that the sharp lower bounds on the stage at which the earliest second nonlinear syzygy occurs is:
\begin{conjecture}  Let $G$ be a simple graph, $I_G$ its edge ideal, and $[k;a_1,...,a_{k-1}]$ be the sequence in \ref{mainquestion1}.  Then the following holds:
\begin{enumerate}
\item  If $a_1=2$, then $a_2\geq 8$.
\item If $a_1=3$, then $a_2\geq 12$.\footnote{This bound may not be sharp.}
\end{enumerate}
\end{conjecture}
These sequences characterize at what stage the width of the resolution of $I_G$ increases, which can be thought of as a measure of the complexity of the resolution through that homological stage.  In this paper we give some restrictions on permissible sequences, and provide several classes of edge ideals partially spanning the set of possible sequences.  Characterizing the types of degree increases in the resolution of ideals of this form provides a tool to help characterize both the algebraic properties of edge ideals and the topological properties of certain flag simplicial complexes.\\
\\
These algebraic questions are equivalent to a question about the topology of flag simplicial complexes:
\begin{question}  Given a flag simplicial complex $\Delta$, and any ordering of the vertices $\{v_1,v_2,...,v_n\}$, consider the chains of nested induced subcomplexes
$$\emptyset=\Delta|_{V_0}\subset \Delta|_{V_1}\subset\Delta|_{V_2}\subset\cdots \subset\Delta|_{V_k}\subset\cdots \subset\Delta|_{V_n}.$$
What can be said about the sequence $a_i:=\min\{k: \dim\homol{i}(V_k)\neq 0\}-i$ for $i\geq 1$?  If the 1-skeleton of $\Delta$ is assumed to be $C_4$ free, what types of sequences $\{a_i\}_{i=1}^{\dim\Delta-1}$ are possible?
\end{question}
We answer this question in Section \ref{polytopes} for edge ideals whose Stanley-Reisner complex can be represented as a regular convex polytope subject to some conditions on their induced subcomplexes.



\section{Algebraic Background}
We introduce some terminology to standardize our notation.
\begin{definition}  The \emph{clique complex of a graph G}, denote $\widehat{G}$, is the simplicial complex on the vertex set of $G$ whose facets are the maximal cliques, or maximal complete subgraphs, of $G$.  The \emph{clique closure of a simplicial complex $\Delta$} denoted $\widehat{\Delta}$, is the complex obtained by closing the complex under the operation of adding a face $\sigma$ to $\Delta$ whenever $\partial\sigma\in\Delta$.
\end{definition}
\begin{remark}  Complexes such that $\Delta=\widehat{\Delta}$ are referred to as either \emph{clique} or \emph{flag} complexes.
\end{remark}
\begin{definition}  Let $I\subset R=\Bbbk[x_1,...,x_n]$ be a square-free monomial ideal, also referred to as a \emph{Stanley-Reisner ideal.}  Then $\Delta_I$, the Stanley-Reisner complex of $I$ is a simplicial complex on vertex set $\{x_1,...,x_n\}$ given by
$$\{\sigma=\{x_{i_1},...,x_{i_r}\}\in\Delta: {\bf m}\nmid x_{i_1}\cdots x_{i_r}\forall {\bf m}\in I\}.$$
\end{definition}
As the Stanley Reisner complex of the ideal $I_G$ is exactly the clique complex of the complement graph $G^c$, properties of the complement graph feature heavily in determinations of the resolutions of $I_G$.  We denote the Stanley Reisner complex of $I_G$ as $\Delta_G$ or $\Gc$ throughout.
\begin{definition}  Let $G$ be a graph with $I_G$ its edge ideal in ring $R=[x_1,...,x_n]$.  Then ${\mathcal F}$, the minimal graded free resolution of $I_G$ is a chain complex of the form
$${\mathcal F}:\;\;\;\cdots\longrightarrow\bigoplus_{j\geq 0} R(-j)^{\betti{i}{j}}\xrightarrow{\varphi_i}\cdots\longrightarrow\bigoplus_{j\geq 0} R(-j)^{\betti{1}{j}}\xrightarrow{\varphi_1}\bigoplus_{j\geq 0} R(-j)^{\betti{0}{j}}\xrightarrow{\varphi_0}R\rightarrow I_G,$$
with $\varphi_i:F_i\rightarrow F_{i-1}$ degree zero maps with entries in ${\mathfrak m}$.
\end{definition}
The ranks of the modules in the resolution of $I_G$ are an invariant of $G$.  These are referred to as the \emph{Betti numbers of $I_G$}, with the Betti numbers of the graded resolution denoted $\betti{i}{j}$ and the Betti numbers of the multigraded resolution are denoted $\betti{i}{m}$, $m$ a square-free monomial.  This relies on a slight abuse of notation -- as written here, we are indexing by the multigraded monomial supported on the multidegree, rather than by the multidegrees themselves.



\section{Jump Sequences and Betti Diagrams}
We formalize our definition of these jump sequences and provide several examples.  We will use the following well-known results in the calculation of our Betti numbers and in the characterization of our graphs and simplicial complexes:
\begin{proposition}  Let $G$ be a graph, $I_G$ its edge ideal, and $\Delta_G$ its Stanley-Reisner complex.  Then the $\Delta_G$ is clique closed, $\Delta_G=\widehat{\Delta_G}$, and its 1-skeleton is the complement graph of G, $\left(\Delta_G\right)_1=G^c$.
\end{proposition}
\begin{proof}  As $I_G$ is generated in degree 2, all minimal nonfaces of the Stanley-Reisner complex are edges.  If the faces in the boundary of a simplex $\partial\sigma$ of dimension greater than 2 are all in $\Delta_G$, then $\sigma\in\Delta_G$.  So $\Delta_G$ is clique-closed.  Every minimal nonface of $\Delta_G$ is an edge in $G$, so the 1-skeleton of $\Delta_G$ is precisely the edges not in $G$.  So $\left(\Delta_G\right)_1=G^c$.
\end{proof}
Using this proposition, we reformulate statements about the induced matching number of $G$, $\Ind{G}$ in terms of properties of $\Delta_G$.
\begin{proposition}\label{indnumber} Let $G$ be a simple graph with edge ideal $I_G$ and Stanley-Reisner complex $\Delta_G$.  The following are equivalent:
\begin{enumerate}
\item $\Ind{G}=k$
\item $\Delta_G$ has the boundary of the \emph{$k$-dimensional cross polytope} $\beta_{k+1}$ as an induced subcomplex, i.e. if $S_0$ is a set consisting of two points,
$$\partial \beta_{k+1}=\overbrace{S_0\ast S_0\ast\cdots\ast S_0}^{k+1}\subseteq \Delta_G,$$
and no $\beta_r$ for $r>k+1$ is an induced subcomplex of $\Delta_G$.
\end{enumerate}
\end{proposition}
\begin{proof}  We prove a slightly stronger statement.  If E is any set of edges of size $r$ in $G$ with the induced graph $G$ on those edges completely disconnected, we have that $\Delta_G$ contains $\beta_{r+1}$, and vice versa.  This is equivalent to proving that the Stanley-Reisner complex of a graph consisting of $r$ disjoint edges is the boundary of the $r$-dimensional cross polytope, as all properties of these complexes rely only on combinatorial data of induced subgraphs and subcomplexes.\\
\\
Without loss of generality, let $G$ be the graph consisting of $r$ disjoint edges, with edge set $E=\{\{x_1,y_1\},\{x_2,y_2\},...,\{x_r,y_r\}\}$ .  By definition, each edge in $G$ is a minimal nonface of $\Delta_G$, and all faces containing at most one vertex in each edge-pair must be in $\Delta_G$.  So the facets ${\mathcal F}$ of $\Delta_G$ must be of the form
$${\mathcal F}=\{\sigma=\{w_1,w_2,...,w_r\}:w_i=x_i\text{ or }w_i=y_i\}.$$
This is precisely the boundary of the $r$-dimensional cross polytope.
\end{proof}
\begin{example}  For the graph $G$ consisting of 3 disjoint edges, we see that $\Delta_G=\partial\beta_3\cong S^2$.
\begin{center}
\begin{tikzpicture}
[scale=0.8, vertices/.style={circle, fill=black, inner sep=1pt}]

\node[vertices, label=below:{$x_1$}] (x1) at (0,0){};
\node[vertices, label=above:{$y_1$}] (y1) at (0.2,1){};

\node[vertices, label=below:{$x_2$}] (x2) at (0.9,0){};
\node[vertices, label=above:{$y_2$}] (y2) at (1.1,1){};

\node[vertices, label=below:{$x_3$}] (x3) at (1.8,0){};
\node[vertices, label=above:{$y_3$}] (y3) at (2,1){};

\foreach \from/\to in {x1/y1,x2/y2,x3/y3}
	\draw[-] (\from)--(\to);

\end{tikzpicture}\;\;\;\;\;\;\;\;\;\;
\begin{tikzpicture}
[scale=0.8, vertices/.style={circle, fill=black, inner sep=1pt}]

\draw [fill=black!5] (0,0)--(1,-1)--(2,0)--(1,1)--cycle;

\node[vertices,label=left:{$x_1$}] (x1) at (0,0){};
\node[vertices,label=below left:{$x_2$}] (x2) at (.85,-.3){};
\node[vertices,label=right:{$y_1$}] (y1) at (2,0){};
\node[vertices,label=above right:{$\;y_2$}] (y2) at (1.15,.3){};
\node[vertices,label=above:{$x_3$}] (x3) at (1,1){};
\node[vertices,label=below:{$y_3$}] (y3) at (1,-1){};

\foreach \from/\to in {x1/x2,x2/y1,x3/x1,x3/y1,x3/x2,y3/y1,x2/y3}
	\draw[-] (\from)--(\to);

\foreach \from/\to in {x3/y2,x1/y2,y2/y3,y1/y2}
	\draw[dashed] (\from)--(\to);

\end{tikzpicture}
\end{center}
\end{example}
\begin{theorem}\label{hochsters} (Hochster's Formula) \cite{Ho77}  Let $I_{\Delta}$ be a square-free monomial ideal in variables $X=\{x_1,...,x_n\}$, with Stanley-Reisner complex $\Delta$.  Then if ${\bf m}$ is a square-free monomial with support $W=\{x_{i_1},...,x_{i_j}\}\subseteq X$ with $\deg(m)=j$, we have
$$\betti{i}{\bf m}(k[\Delta])= \dim \widetilde{H}_{j-i-1}(\Delta|_{W}, k),$$
where $\Delta|_W$ is the induced subcomplex of $\Delta$ on vertices in $W$.
\end{theorem}
Combining these two, we have Hochster's formula for the Betti numbers of edge ideals.
\begin{proposition}\label{graphcliques}Let $G$ be a simple graph on vertex set $[n]=\{1,2,...,n\}$ with edge set $E$, and let $I_G=(x_ix_j: \{i,j\}\in E)\subseteq R=k[x_1,...,x_n]$ be the edge ideal of $G$.  Then the Stanley-Reisner complex of $I_G$, denoted $\Delta(I_G)$, is given by
$$\Delta(I_G)= \Gc,$$
the clique closure of the complement graph of $G$ in $[n]$.  So
$$\betti{i}{\bf m}(I_G)=\dhomol{j-i-1}(\Gc,\Bbbk).$$
\end{proposition}
These will form the primary basis of our Betti number calculations.  Tying these statements all together, we have the following corollary that provides a helpful characterization of graphs and their complements.
\begin{corollary}  Let $G$ be a simple graph, $G^c$ its complement graph and $I_G$ its edge ideal.  The following are equivalent:
\begin{enumerate}
\item $\Ind{G}=1$,
\item $G^c$ has no induced 4-cycles, and
\item $\betti{2}{4}(I_G)=0$.
\end{enumerate}
\end{corollary}
\begin{proof}  We have that $(1)\Leftrightarrow (2)$ from Proposition \ref{indnumber} with $k=1$.  Using Proposition \ref{graphcliques} for graded modules, we have that
$$\betti{2}{4}(I_G)=\sum_{\substack{W\subseteq V\\
|W|=4}
}
\dhomol{1}(\Gc|_W,\Bbbk).$$
This Betti number is precisely nonzero when the $\Gc|_W$ has no cycles of length 4 in the complement graph $G$ if we restrict to any set of vertices of $G$ of size 4.  By Proposition \ref{indnumber}, this is true precisely when there are no pairs of induced disjoint edges in our original graph, as a 4-cycle is the one dimensional cross polytope.
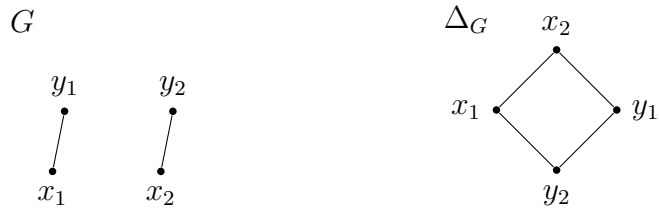
\begin{center}
\begin{tikzpicture} [scale=.8,auto=left, fill=black, vertices/.style={circle, fill=black, inner sep=1pt}]

\node (G) at (-.5,2.5) {$G$};
\node[vertices, label=below:{$x_1$}] (x1) at (0,0){};
\node[vertices, label=above:{$y_1$}] (y1) at (0.2,1){};

\node[vertices, label=below:{$x_2$}] (x2) at (1.8,0){};
\node[vertices, label=above:{$y_2$}] (y2) at (2,1){};

\foreach \from/\to in {x1/y1,x2/y2}
	\draw [-] (\from)--(\to);

\end{tikzpicture}
\;\;\;\;\;\;\;\;\;\;\;\;\;\;\;\;\;\;\;\;\;\;\;\;\;
\begin{tikzpicture}
[scale=0.8, vertices/.style={circle, fill=black, inner sep=1pt}]

\node (G) at (-.5,1.5) {$\Delta_G$};
\node[vertices,label=left:{$x_1$}] (x1) at (0,0){};
\node[vertices,label=right:{$y_1$}] (y1) at (2,0){};

\node[vertices,label=above:{$x_2$}] (x2) at (1,1){};
\node[vertices,label=below:{$y_2$}] (y2) at (1,-1){};

\foreach \from/\to in {x1/y2,x1/x2,y1/y2,y1/x2}
	\draw[-] (\from)--(\to);

\end{tikzpicture}

\end{center}
\end{proof}
Given a free resolution ${\mathcal F}$ of $I_G$, we can keep track of the twists in degree of the modules of our resolution via the $\betti{i}{j}(I_G)$.  This information is placed into the \emph{Betti table of the resolution of $I_G$} as a bookkeeping device for the ranks of these syzygy modules, the $\betti{i}{j}$.  We denote the Betti diagram using the convention of Macaulay 2, shifting the displayed degrees in the ith homological stage down by i rows.



\begin{center}
\begin{tikzpicture} [scale=.8,auto=left,betti/.style={
execute at begin cell=\node\bgroup,
execute at end cell=\egroup;,
execute at empty cell=\node{$\cdot$};
},
vertices/.style={circle, fill=black, inner sep=0pt}]

\node [vertices] (r1) at (-5,.6){};
\node [vertices] (r2) at (-5,-.1){};
\node [vertices] (r3) at (-4,-.1){};
\node [vertices] (r4) at (-4,-.9){};
\node [vertices] (r5) at (-2.8,-.9){};
\node [vertices] (r6) at (-2.8,-1.4){};
\node [vertices] (r7) at (-2.3,-1.4){};
\node [vertices] (r8) at (-2.3,-1.8){};
\node [vertices] (r9) at (-1.7,-1.8){};
\node [vertices] (r10) at (-1.7,-2.2){};
\node [vertices] (r11) at (-1.7,-2.8){};
\node [vertices] (r12) at (-.5,-2.8){};

\foreach \from/\to in {r1/r2, r2/r3, r3/r4,r4/r5,r5/r6,r10/r11,r11/r12}
	\draw[-] (\from)--(\to);
\foreach \from/\to in {r6/r7,r7/r8,r8/r9,r9/r10}
	\draw[dotted] (\from)--(\to);
\matrix [betti] {
-&0&1&2&3&4&$\cdots$&i&$\cdots$&n-2&n-1\\
total: &1&$\beta_{1}$&$\beta_{2}$&$\beta_{3}$&$\beta_{4}$&$\cdots$&$\beta_{i}$&$\cdots$&$\beta_{n-2}$&$\beta_{n-1}$\\
0: &1&&&&&&&&&\\
1:& &$\betti{1}{2}$&$\betti{2}{3}$&$\betti{3}{4}$&$\betti{4}{5}$&$\cdots$&$\betti{i}{i+1}$&$\cdots$&$\betti{n-2}{n-1}$&$\betti{n-1}{n}$\\
2:& &&$\betti{2}{4}$&$\betti{3}{5}$&$\betti{4}{6}$&$\cdots$&$\betti{i}{i+2}$&$\cdots$&$\betti{n-2}{n}$&\\
$\vdots$&&&\\
k:& &&&&$\betti{k}{2k}$&$\cdots$&$\betti{i}{i+k}$&$\cdots$&&\\
};
\end{tikzpicture}
\end{center}

\begin{remark}  The first possible nonzero Betti number in each row is $\betti{i}{2i}$.  This follows immediately from the lcm lattice and from the fact that all generators of $I_G$ are of degree two.  For any fixed $s\geq 2$, it may be the case that $\betti{s}{2s}=0$ but later entries in the row [Betti numbers of form $\betti{i}{i+s}$ for $i>s$] may be nonzero.  This gives us a staircase of sorts walking down the left edge of the Betti table, with each step of length at least one.\\
\\
On the other side of the Betti diagram, no Betti number can occur in degrees greater than $n$, so $\betti{i}{j}=0$ for all $j>n$.
\end{remark}
This idea of a \emph{staircase} walking down the highest degree Betti numbers leads to the definition of our jump sequence.  If we have a Betti table of the following form,




\begin{center}
\begin{tikzpicture} [scale=.8,auto=left,betti/.style={
execute at begin cell=\node\bgroup,
execute at end cell=\egroup;,
execute at empty cell=\node{$\cdot$};
},
vertices/.style={circle, fill=black, inner sep=0pt}]

\node [vertices] (r1) at (-7,.2){};
\node [vertices] (r2) at (-5.1+1.4,.2){};
\node [vertices] (r3) at (-5.1+1.4,-.7){};
\node [vertices] (r4) at (-1.6+1.4,-.7){};
\node [vertices] (r5) at (-1.6+1.4,-1.5){};
\node [vertices] (rk) at (1+1.4,-1.5){};
\node [vertices] (r6) at (1.7+1.4,-1.5){};
\node [vertices] (r7) at (1.7+1.4,-1.9){};
\node [vertices] (r8) at (2.2+1.4,-1.9){};
\node [vertices] (r9) at (2.2+1.4,-2.3){};
\node [vertices] (rj) at (3+1.4,-2.3){};
\node [vertices] (r10) at (5.6+1.4,-2.3){};
\node [vertices] (r11) at (5.6+1.4,-3.3){};
\node [vertices] (r12) at (5.9+1.4,-3.3){};

\foreach \from/\to in {r1/r2,r2/r3,r3/r4,r4/r5,r5/rk,rj/r10,r10/r11,r11/r12}
	\draw[-] (\from)--(\to);
\foreach \from/\to in {rk/r6,r6/r7,r7/r8,r8/r9,r9/rj}
	\draw[dotted] (\from)--(\to);

\node [vertices] (s1) at (-8.4+1.4,0){};
\node [vertices] (s2) at (-5.1+1.4,0){};
\draw[dashed] (s1)--(s2) node[midway, below] {$a_1$};

\node [vertices] (s3) at (-8.4+1.4,-.9){};
\node [vertices] (s4) at (-1.6+1.4,-.9){};
\draw[dashed] (s3)--(s4) node[midway, below] {$a_2$};

\node [vertices] (s5) at (-8.4+1.4,-1.7){};
\node [vertices] (s6) at (1.7+1.4,-1.7){};
\draw[dashed] (s5)--(s6);

\node [vertices] (s7) at (-8.4+1.4,-2.1){};
\node [vertices] (s8) at (2.2+1.4,-2.1){};
\draw[dashed] (s7)--(s8);

\node [vertices] (s3) at (-8.4+1.4,-2.5){};
\node [vertices] (s4) at (5.6+1.4,-2.5){};
\draw[dashed] (s3)--(s4) node[midway, below] {$a_{k-1}$};

\matrix [betti] {
-&0&1&2&3&$a_1+1$&$\cdots$&$a_2+1$&$\cdots$&$a_{k-1}-1$&$a_{k-1}$&$a_{k-1}+1$\\
total: &1&$\beta_{1}$&$\beta_{2}$&$\cdots$&$\beta_{a_1+1}$&$\cdots$&$\beta_{a_2+1}$&$\cdots$&$\beta_{a_{k-1}-1}$&$\beta_{a_{k-1}}$&$\cdots$\\
0: &1&&&&&&&&&\\
1:& &$\betti{1}{2}$&$\betti{2}{3}$&$\ast$&$\ast$&$\ast$&$\ast$&$\ast$\\
2:&\;&\;&\;&\;&$\betti{a_1+1}{a_1+3}$&$\ast$&$\ast$&$\ast$&$\ast$\\
3:&\;&\;&\;&\;&\;&\;&$\betti{a_2+1}{a_2+4}$&$\ast$&$\ast$&$\ast$&$\ast$\\
$\vdots$&\;&\;&\;&\;&\;&\;&\;&\;&$\ast$&$\ast$\\
k:& \;&\;&\;&\;&\;&\;&\;&\;&\;&\;&$\betti{a_{k-1}+1}{s}$\\
};
\end{tikzpicture}
\end{center}
where $s=a_{k-1}+k+1$, with $\betti{a_i+1}{a_i+i+1}\neq 0$ and all Betti numbers below the line are zero, we say that $G$ (or $I_G$) has the jump sequence ${\bf a}_G=[k;a_1,a_2,...,a_k],$ a sequence of increasing positive integers.  More formally, we make the following definition.
\begin{definition}  Let $G$ a graph, $I_G$ its edge ideal, and $\betti{i}{j}=\betti{i}{j}(I_G)$ as above.  If $I_G$ has $\reg(I_G)=k+1$, then $I_G$ has a \emph{jump sequence of length k-1} of the form
$${\bf a}_G=[k;a_1,...,a_{k-1}],$$
where $a_r=\min\{ i : \betti{i}{i+r+1}\neq 0\}-1$.  If $I_G$ has a linear resolution, we say its jump sequence is ${\bf a}=[1;\emptyset]$.
\end{definition}
First, we provide a few examples to motivate the use of this definition.
\begin{example}  Let $G$ be the graph of the anticycle of length n, e.g. $G^c=C_n$.  The Betti diagram is of the form:




\begin{center}
\begin{tikzpicture} [betti/.style={
execute at begin cell=\node\bgroup,
execute at end cell=\egroup;,
execute at empty cell=\node{$\cdot$};
},
vertices/.style={circle, fill=black, inner sep=0pt}] 

\node [vertices] (s1) at (-2.3,-.9){};
\node [vertices] (s2) at (3,-.9){};
\draw[-] (s1)--(s2) node[midway, below] {$a_1=n-3$};

\matrix [betti] {
-&0&1&2&3&$\cdots$&$n-4$&$n-3$&$n-2$\\
total: &1&$\beta_1$&$\beta_2$&$\beta_3$&$\cdots$&$\beta_{n-4}$&$\beta_{n-3}$&$\beta_{n-2}$\\
0: &1&&&&$\cdots$&&&\\
1:& &$\ast$&$\ast$&$\ast$&$\cdots$&$\ast$&$\ast$&\\
2: &\;&\;&\;&\;&\;&\;&\;&1\\
};
\end{tikzpicture}
\end{center}
So all singleton sequences are possible, with jump sequence $[2;n]$ for edge ideal of the complement graph of the $n+3$ cycle.
\end{example}


\begin{example}\label{balloon} Let $G$ be the graph with edge ideal
\begin{equation*}
\begin{split}
I_G&=(x_1x_3, x_1x_4, x_2x_4, x_2x_5, x_3x_5, x_2x_6, x_3x_6, x_4x_6,x_3x_7, x_4x_7, x_5x_7, x_1x_8,x_4x_8,\\
&x_5x_8, x_6x_8,x_1x_9,x_2x_9, x_5x_9, x_6x_9, x_7x_9, x_1x_{10}, x_2x_{10}, x_3x_{10},x_7x_{10},x_8x_{10},
x_6x_{11},\\
&x_7x_{11}, x_8x_{11}, x_9x_{11}, x_{10}x_{11},x_1x_{12}, x_2x_{12},x_3x_{12}, x_4x_{12},x_5x_{12}, x_{11}x_{12}).
\end{split}
\end{equation*}
This graph $G$ has $\Ind{G}=1$, and a complement clique complex isomorphic to the icosahedron.  Its Stanley Reisner complex and Betti diagram have the following forms:




\begin{center}
\begin{tikzpicture}[scale=0.8, vertices/.style={circle, fill=black, inner sep=0pt}]

\node (D) at (0,1){$\Delta_G$};
\node[vertices] (w1) at (1.5,1.3) {};
\node[vertices] (w2) at (1.5,-2.3) {};

\node[vertices] (x1) at (0,0) {};
\node[vertices] (x2) at (.8,-.5) {};
\node[vertices] (x3)  at (2.2,-.5) {};
\node[vertices] (x4)  at (3,0){};
\node[vertices] (x5)  at (1.4,.3){};

\node[vertices] (y1) at (0,0-1) {};
\node[vertices] (y2) at (.8,-.5-1) {};
\node[vertices] (y3)  at (2.2,-.5-1) {};
\node[vertices] (y4)  at (3,0-1){};
\node[vertices] (y5)  at (1.4,.3-1){};

\foreach \from/\to in {x1/x2,x2/x3,x3/x4,y1/y2,y2/y3,y3/y4,x1/y1,x2/y2,x3/y3,x4/y4,x1/w1,x2/w1,x3/w1,x4/w1,y1/w2,y2/w2,y3/w2,y4/w2,y1/x2,y2/x3,y3/x4}
	\draw[-] (\from)--(\to);
\foreach \from/\to in {x1/x5,x4/x5,y1/y5,y4/y5,x5/y5,x5/w1,y5/w2,y4/x5,y5/x1}
	\draw[dotted] (\from)--(\to);

\end{tikzpicture}\;\;\;\;\;
\begin{tikzpicture} [betti/.style={
execute at begin cell=\node\bgroup,
execute at end cell=\egroup;,
execute at empty cell=\node{$\cdot$};
}] 
\matrix [betti] {
-&0&1&2&3&4&5&6&7&8&9\\
total: &1&36&160&327&412&412&327&160&36&1\\
0: &1&&&&&&&&&\\
1:&&36&160&315&300&112&12&&&\\
2:&&&&12&112&300&315&160&36&\\
3:&&&&&&&&&&1\\
};
\end{tikzpicture}
\end{center}




\begin{center}

\end{center}
This has jump sequence ${\bf a}=[3;2,8]$.
As our $\Delta_G$ can be represented as a (regular) convex simplicial polytope, we have the ring $\Bbbk[\Delta_G]$ is Gorenstein and shellable.  Hence, restricting to the classes of Gorenstein or shellable clique complexes  cannot provide a sharper bound than $\reg(I_G)\leq 5$ at best.  Example \ref{600cell} is a convex simplicial 4-polytope so, in fact, regularity bounds in these classes cannot be sharpened below $\reg(I_G)\leq 5$.  There exist Gorenstein graphs $G$ with $\Ind{G}=1$, $\reg(I_G)=4$ and arbitrarily high projective dimension, which we construct in Section \ref{highregclasses}.
\end{example}
Sometimes it will be more convenient to discuss the number of homological stages between a pair of degree jumps instead of the sequences of homological degrees these jumps occur at.  We use the \emph{relative jump sequence} in this case.
\begin{definition}  Let $G$ be a graph, $I_G$ its edge ideal, and ${\bf a}_G$ its jump sequence.  Then the \emph{relative jump sequence of $G$} is
$${\bf r}_G=[k; r_1,r_2,...,r_{k-1}],$$
where $r_1=a_1$, $r_2=a_2-a_1$, $r_3=a_3-a_2$, etc.
\end{definition}
The advantage of working with the relative jump sequence is that the topological interpretation of the $r_i$ is more straightforward than that of the $a_i$.  If we have a subset of vertices $W\subset V$ of minimal size for which $\dim\widetilde{H}_{i-2}(\Gc|_W)\neq 0$, then $r_i+2$ is the number of vertices that must be added to W to find a set of vertices $W'$ with $\dim\widetilde{H}_{i-1}(\Gc|_{W'})\neq 0$.  It might be the case that a particular W has no subset of size $r_i+2$ for which the rank of this homology is nonzero, but there exist at least \emph{one} subset $W$ such that we can find such a $W'$.\\
\\
In the edge ideal introduced in Example \ref{balloon}, the jump sequence was ${\bf a}=[3;2,8]$ and the relative jump sequence is ${\bf r}=[3;2,6]$.
\begin{example}\label{600cell} The first example found with $\Ind{G}=1$ and $\reg(I_G)=5$ was the complement graph of the 1-skeleton of the 600-cell [or hexacosichoron, $\Delta_H$] which we denote $G_{H}=(\Delta_H)_1^c$.  This was first noted by Nevo and Peeva (private communication.)  This is a graph on $120$ vertices with 6420 edges, with an edge ideal $I_{G_H}$ with regularity 5.\\
\\
Since $\Delta_H$ is equal to the clique closure of its 1-skeleton, and as its smallest induced cycle is of length 5, we have that $\Delta_H=\dim\widehat{G_H^c}$ with induced matching size 1.  The edge ideal $I_{G_H}$ then has jump sequence ${\bf a}=[3;2,8,115]$.\\
\\
We have that the $\Ind{G}=1$ from the smallest cycles being of length 5, that the smallest induced simplicial 2-sphere in $\Delta_H$ is an icosahedron (one for each vertex of $\Delta_H$, as an icosahedron formed of 20 tetrahedral cells lie around each vertex,) and that the entire complex is a simplicial 3-sphere on 120 vertices.  So our number of vertices involved in the minimal reduced homology generators are 5, 12, and 120 respectively -- giving rise to jump sequence ${\bf a}=[3;2,8,115]$.
\end{example}



\section{Bounds on Jump Sequences}\label{bounds}

We now have the terminology to rephrase Question \ref{mainquestion1} more precisely:
\begin{question}  Given any integer $k$ and a strictly increasing sequence $\{a_1,...,a_{k-1}\}$, does there exists an edge ideal $I_G$ with jump sequence $[k;a_1,...,a_k]$?
\end{question}
For arbitrary sequences, this question is answered negatively by the following theorem.
\begin{theorem}\label{maintheorem}  Let $I_G$ be an edge ideal with jump sequence ${\bf a}=[k;a_1,a_2,...,a_{k-1}]$ and relative jump sequence $[k;r_1,r_2,...,r_{k-1}]$.  Then equivalently
$$2 a_1\leq a_2\text { or }r_1\leq r_2.$$
\end{theorem}
To standardize terminology, let $n(v)$ denote the set of neighbors of a vertex $v\in V$, with the degree of a vertex $v\in V$ equal to the size of this set, $\deg(v)=|n(v)|$.\\
\\
It is convenient to reduce this problem for general jump sequences to a problem on jump sequences of length 2.  The following lemma allows us to go even further:



\begin{lemma}[Subgraph Reduction Lemma]\label{reductionlemma} Let $G$ be a graph such that $I_G$ has jump sequence $[k;a_1,a_2,...,a_n],$ with $k\geq 3$.  There exists an induced subgraph $H$ of $I_G$ on vertex set $W\subset V$ of size $|W|=a_2+4$ such that $I_H$ has jump sequence $[3; a_1',a_2]$, and $a_1\leq a_1'$, and such that $H$ has no induced subgraphs $W'\subseteq W$ such that $\dim\widetilde{H}_2((\Delta_H)|_{W'},k)\neq 0$.
\end{lemma}
\begin{proof}  This follows from the definition of our jump sequences,
\begin{equation*}a_r=\min \{ i : \betti{i}{j}\neq 0\text{ and } j-i=r\} -1.\end{equation*}
So $a_2$ is the smallest Betti index such that $\betti{a_2+1}{a_2+4}\neq 0$.  Via Theorem \ref{hochsters}, this is equivalent to saying there exists some subset of vertices $W\subset V$ such that $|W|=a_2+4$ with the dimension of the reduced second homology of $(\Delta_G)|_W=(\Delta_{G|_W})$ nonzero, and no smaller subset $W'\subset V$ will give us nonzero $\widetilde{H}_2$.  Let ${\mathcal W}$ denote the set of all such $W\subseteq V$.\\
\\
Among such subsets $W$, choose one such that the size of the smallest induced cycle in $(\Delta_G)|_W$, i.e. given any pair $W,W'\in{\mathcal W}$, if $c$ is an induced cycle of minimal size in $\Delta_W$, then there exists a cycle $c'\in W'$ in $\Delta|_{W'}$ such that the $|c'|\leq |c|$.  Let $H$ be our induced subgraph of $G$ on vertex set $W$.\\
\\
By construction, the minimal cycles of $G^c$ are of length smaller than or equal to the minimal cycles of $H^c$, and the minimal vertex sets of $G$ on which induced subcomplexes of $\Gc$ have $\widetilde{H}_2\neq 0$ are of the same size as those of $\widehat{H^c}$, so for $G$ we have the jump sequence of $I_G$ is $[k;a_1,a_2,...,a_{k-1}]$ and the jump sequence of $I_H$ is $[3;a_1',a_2]$ with $a_1\leq a_1'$.
\end{proof}
The upshot of Lemma \ref{reductionlemma} is the following:  If we can prove the Theorem \ref{maintheorem} holds for graphs of the form $H$, $|H|=a_2+5$, with jump sequence $a=[3;a_1',a_2]$, we will have that it holds for all graphs G, using the inequalities $2a_1\leq 2a_1'\leq a_2$.  We can now reduce the general problem to cases where Stanley-Reisner complex of the graph is our minimal generator of nonzero $\widetilde{H}_2$, and removing any vertex $v\in H$ will drop the dimension of the homology by at least one.
\begin{proof}[Proof of Theorem \ref{maintheorem}.]  Without loss of generality, assume $G$ is a graph of the form introduced in Lemma \ref{reductionlemma}.  If $a_1=1$, then $a_2\geq 2$ by $I_G$ generated in degree 2.  So we consider only the case where $a_1\geq 2$.  We note that in this case, the size of the vertex set of $G$ is $|V|=a_2+4$.\\
\\
Assume by contradiction that $2a_1> a_2$, and let $W$ be the vertex set minimal induced cycle $c\in G^c$.  We have that the length of $c=|W|=a_1+3$.  Choose two nonadjacent vertices $\{v,w\}$ in this cycle $c$.  Orient the cycle from $v$ to $w$, then back to $v$, and partition the vertices in $c$ into sets $V_1,V_2\subset W$, such that
\begin{align*}
V_1&=\{v,v_1,v_2,...,v_k,w\}\\
V_2&=\{w,w_1,w_2,...,w_{a_1-k+1},v\}
\end{align*}
where $\{v,v_1,v_2,...,v_k,w\}$ and $\{w,w_1,w_2,...,w_{a_1-k+1},v\}$ are vertices in the oriented paths from $v$ to $w$, then from $w$ to $v$, endpoints inclusive.
Let $K=V\backslash W$, and $K_1=K\cup V_1$, $K_2=K\cup V_2$.

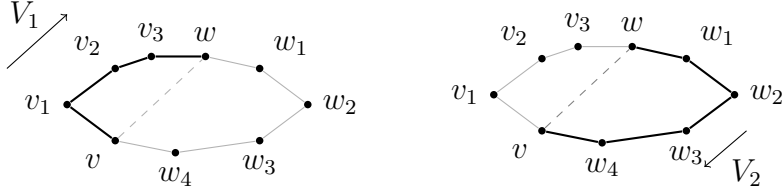
\begin{center}
\begin{tikzpicture}
[scale=.8,auto=left,vertices/.style={circle, fill=black, inner sep=1pt}]
\node (V1) at (-.7,1.5){$V_1$};

\draw [->] (-1,.6)--(0,1.5);

\node [vertices,label=left:{$v_1$}] (a) at (0,0){};
\node [vertices,label=above left:{$v_2$}] (b) at (1-.2,.6){};
\node [vertices,label=above:{$v_3$}] (c) at (1.4, .8){};
\node [vertices,label=above:{$w$}] (d) at (2.3, .8){};
\node [vertices,label=above right:{$w_1$}] (e) at (2+.2+1,.6){};
\node [vertices,label=right:{$w_2$}] (f) at (3+1,0){};
\node [vertices,label=below:{$w_3$}] (g) at (2+.2+1,-.6){};
\node [vertices,label=below:{$w_4$}] (h) at (1.8,-.8){};
\node [vertices,label=below left:{$v$}] (i) at (1-.2,-.6){};

\foreach \from/\to in {a/b,b/c,c/d,a/i}
	\draw [-,thick] (\from) -- (\to);

\foreach \from/\to in {d/e,e/f,f/g,g/h,h/i}
	\draw [-, black!30] (\from) -- (\to);
\draw [dashed, black!30] (i)--(d);

\end{tikzpicture}
\;\;\;\;\;\;
\begin{tikzpicture}
[scale=.8,auto=left,vertices/.style={circle, fill=black, inner sep=1pt}]

\node (V2) at (4.2,-1.3){$V_2$};
\draw [->] (4.2,-.6)--(3.5,-1.2);
\node [vertices,label=left:{$v_1$}] (a) at (0,0){};
\node [vertices,label=above left:{$v_2$}] (b) at (1-.2,.6){};
\node [vertices,label=above:{$v_3$}] (c) at (1.4, .8){};
\node [vertices,label=above:{$w$}] (d) at (2.3, .8){};
\node [vertices,label=above right:{$w_1$}] (e) at (2+.2+1,.6){};
\node [vertices,label=right:{$w_2$}] (f) at (3+1,0){};
\node [vertices,label=below:{$w_3$}] (g) at (2+.2+1,-.6){};
\node [vertices,label=below:{$w_4$}] (h) at (1.8,-.8){};
\node [vertices,label=below left:{$v$}] (i) at (1-.2,-.6){};

\foreach \from/\to in {a/b,b/c,c/d,a/i}
	\draw [-,black!30] (\from) -- (\to);

\foreach \from/\to in {d/e,e/f,f/g,g/h,h/i}
	\draw [-, thick] (\from) -- (\to);
\draw [dashed, black!50] (i)--(d);

\end{tikzpicture}
\end{center}
We consider the complexes $\Delta_{K_1}=\left(\Delta_G\right)|_{K_1}$ and  $\Delta_{K_2}=\left(\Delta_G\right)|_{K_2}$, on these vertex sets $K_1$ and $K_2$.  We also let $\Delta_{K'}=\Delta_{K_1}\cap\Delta_{K_2}$, and $K'=K_1\cap K_2= K\cup\{v,w\}$.
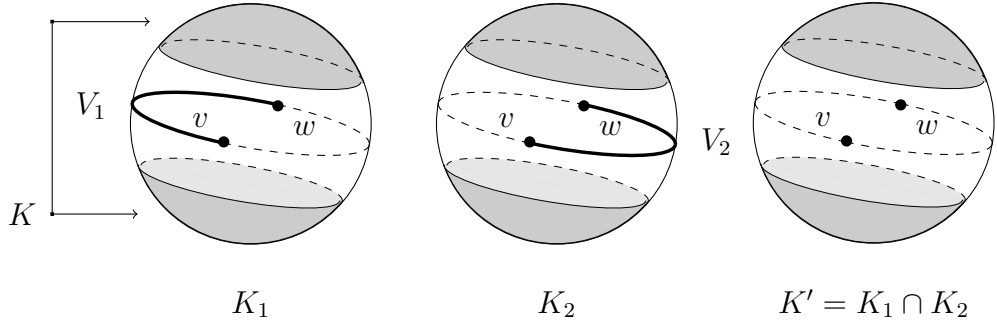
\begin{center}
\begin{tikzpicture}
[scale=.8,auto=left,vertices/.style={circle, fill=black, inner sep=1.5pt}]

\node [draw, fill=black, inner sep=0pt](j) at (-3.3,1.7){};
\node [draw, fill=black, inner sep=0pt](k) at (-3.3,-1.5){};
\node (l) at (-1.5,1.7){};
\node (m) at (-1.7,-1.5){};
\node (n) at (-3.8, -1.5) {$K$};

\foreach \from/\to in {j/k}
	\draw [-] (\from)--(\to);
\foreach \from/\to in {j/l,k/m}
	\draw [->] (\from)--(\to);

\draw (0,0) circle (2);
\draw [rotate=-10, dashed](2,0) arc (0:80:2 and .4);
\draw [line width=1.3pt, rotate=-10, -](-2,0) arc (180:80:2 and .4);
\draw [line width=1.3pt, rotate=-10] (-2,0) arc (180:260:2 and .4);
\draw [rotate=-10, dashed] (2,0) arc (360:260:2 and .4);

\draw [fill=black!20, rotate=-10] (1.73,.99) arc (30:150: 2 and 2);

\draw [fill=black!20, rotate=-10] (1.73,-.99) arc (330:210: 2 and 2);

\filldraw [fill=black!10,line width=0 pt, rotate=-10,dashed] (1.72,-1) arc(0:180:1.7 and .3);
\filldraw [fill=black!10,line width=0 pt,rotate=-10] (-1.72,-1) arc(180:360:1.7 and .3);
\filldraw [fill=black!20,line width=0 pt,rotate=-10,dashed] (1.72,1) arc(0:180:1.7 and .3);
\filldraw [fill=black!20,line width=0 pt,rotate=-10] (-1.72,1) arc(180:360:1.7 and .3);

\node[vertices, label=below right:{$w$}] (w) at (.45,.3){};
\node[vertices, label=above left:{$v$}] (v) at (-.45,-.3){};

\node [scale=.7, label=left:{$V_1\;\;$}](hem) at (-1.8,.3){};

\node (K1) at (0,-3) {$K_1$};

\end{tikzpicture}
\;\;\;\;\;
\begin{tikzpicture}
[scale=.8,auto=left,vertices/.style={circle, fill=black, inner sep=1.5pt}]


\draw (0,0) circle (2);
\draw [line width=1.3pt,rotate=-10](2,0) arc (0:80:2 and .4);
\draw [rotate=-10, dashed](-2,0) arc (180:80:2 and .4);
\draw [rotate=-10, dashed] (-2,0) arc (180:260:2 and .4);
\draw [line width=1.3pt,rotate=-10] (2,0) arc (360:260:2 and .4);

\draw [fill=black!20, rotate=-10] (1.73,.99) arc (30:150: 2 and 2);

\draw [fill=black!20, rotate=-10] (1.73,-.99) arc (330:210: 2 and 2);

\filldraw [fill=black!10,line width=0 pt, rotate=-10,dashed] (1.72,-1) arc(0:180:1.7 and .3);
\filldraw [fill=black!10,line width=0 pt,rotate=-10] (-1.72,-1) arc(180:360:1.7 and .3);
\filldraw [fill=black!20,line width=0 pt,rotate=-10,dashed] (1.72,1) arc(0:180:1.7 and .3);
\filldraw [fill=black!20,line width=0 pt,rotate=-10] (-1.72,1) arc(180:360:1.7 and .3);

\node[vertices, label=below right:{$w$}] (w) at (.45,.3){};
\node[vertices, label=above left:{$v$}] (v) at (-.45,-.3){};

\node [scale=.7, label=right:{$\;\;V_2$}](hem) at (1.8,-.3){};

\node (K2) at (0,-3) {$K_2$};

\end{tikzpicture}
\begin{tikzpicture}
[scale=.8,auto=left,vertices/.style={circle, fill=black, inner sep=1.5pt}]


\draw (0,0) circle (2);
\draw [rotate=-10, dashed](2,0) arc (0:80:2 and .4);
\draw [rotate=-10, dashed](-2,0) arc (180:80:2 and .4);
\draw [rotate=-10, dashed] (-2,0) arc (180:260:2 and .4);
\draw [rotate=-10, dashed] (2,0) arc (360:260:2 and .4);

\draw [fill=black!20, rotate=-10] (1.73,.99) arc (30:150: 2 and 2);

\draw [fill=black!20, rotate=-10] (1.73,-.99) arc (330:210: 2 and 2);

\filldraw [fill=black!10,line width=0 pt, rotate=-10,dashed] (1.72,-1) arc(0:180:1.7 and .3);
\filldraw [fill=black!10,line width=0 pt,rotate=-10] (-1.72,-1) arc(180:360:1.7 and .3);
\filldraw [fill=black!20,line width=0 pt,rotate=-10,dashed] (1.72,1) arc(0:180:1.7 and .3);
\filldraw [fill=black!20,line width=0 pt,rotate=-10] (-1.72,1) arc(180:360:1.7 and .3);

\node[vertices, label=below right:{$w$}] (w) at (.45,.3){};
\node[vertices, label=above left:{$v$}] (v) at (-.45,-.3){};

\node (K2) at (0,-3) {$K'=K_1\cap K_2$};

\end{tikzpicture}
\end{center}
\begin{lemma}  We have that
\begin{enumerate}
\item $\Delta_{K'}=\Delta_{K_1}\cap \Delta_{K_2}=\Delta_{K_1\cap K_1}=\left(\Delta_G\right)|_{K'}\text{  and}$
\item $\Delta_G=\Delta_{K_1}\cup\Delta_{K_2}.$
\end{enumerate}
\end{lemma}
\begin{proof}  The first fact follows immediately from properties of induced subcomplexes, as for general simplicial complexes $\Delta$ with sets of vertices $S,T\subseteq V$, we have
$$\Delta|_{S\cap T}=\Delta|_S\cap \Delta|_T.$$
The second is not true for all simplicial complexes.  For complexes $\Delta$ of the form above, we wish to show that every $\sigma\in\Delta$ is in either the induced subcomplex $\Delta_{K_1}$ or $\Delta_{K_2}$.  It is sufficient to show this for edges of $\Delta$, as clique closure of $\Delta$ finishes the argument.  It is clear that any edges on vertices entirely contained in $K_1, or K_2$ are in $\Delta_{K_1}\cup \Delta_{K_2}$.  An edge running between the two (but not contained fully in either) would have to be of the form $\{v_i, w_j\}$ for $v_i\in V_1, w\in V_2$.  However, by construction of $c$ as a minimal cycle, no such chords exist.  The result follows.
\end{proof}
With this characterization in hand, we may take a Mayer-Vietoris sequence to finish the proof of Theorem \ref{maintheorem}.
\begin{equation*}
\cdots\rightarrow H_2(\Delta_{K_1})\oplus H_2(\Delta_{K_2})\rightarrow H_2(\Delta_G)\xrightarrow{\partial} H_1(\Delta_{K'})\rightarrow H_1(\Delta_{K_1})\oplus H_1(\Delta_{K_2})\rightarrow\cdots
\end{equation*}
As we had assumed G had no proper induced subgraphs $G'$ with $H_2(\Delta_{G'})\neq 0$, we have the leftmost term is zero.  So the map $\partial: H_2(\Delta_G)\rightarrow H_1(\Delta_{K'})$ is injective, and $\dim H_1(\Delta_{K'})\neq 0$.  This subset $K'$ is of size
$$|K|+2= |V|-|W|+2= (a_2+4)-(a_1+3)+2=a_2-a_1+3.$$
By assumption, $2a_1> a_2$, we have that $|K'|<2a_1-a_1+3=a_1+3$.  So $|K'|$ is a set of vertices strictly smaller than those of $c$ generating a nonzero first homology, which gives us our desired contradiction.
\end{proof}
\begin{remark}  These statements on Betti numbers are really statements about simplicial topology.  This theorem is equivalent to following statement.  Let $\Delta$ be any flag complex on $n$ vertices with $\dim H_2(\Delta)\neq 0$ and $\dim H_2(\Delta\backslash v)=0$ for any vertex $v\in\Delta$.  Given any ordering of the vertices of $\Delta$, $V=\{v_1,...,v_n\},$ we will let $V_i:=\{v_1,...,v_i\}$.  Consider the chain of inclusions of the induced simplicial complexes,
$$\emptyset=\Delta|_{V_0}\subset \Delta|_{V_1}\subset\Delta|_{V_2}\subset\cdots \subset\Delta|_{V_k}\subset\cdots \subset\Delta|_{V_n}.$$
Across all such chains, choose one such that $k$, the first index for which $\dim H_1(\Delta|_{V_k})\neq 0$, is minimal.  Then we have the total number of vertices in our complex must be at least 2k.
\end{remark}
\begin{example}  It is not the case that all minimal homology generators of $\homol{2}$ can be chosen to be spheres, or that given any generator $K$ of nonzero $\homol{1}$ that the vertices can be partitioned into hemispheres $K_1$ and $K_2$ in such a way that $K_1\cap K_2=K$ and $K_1\cap K_2=\Delta_G$.  The 1-skeleton of the following complex $\widehat{G^c}$ is an example of when this partitioning fails.
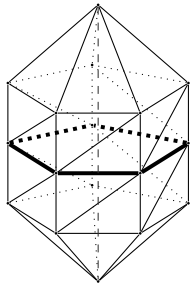
\begin{center}
\begin{tikzpicture}[scale=0.8, vertices/.style={circle, fill=black, inner sep=0pt}]
\node[vertices] (w1) at (1.5,1.3) {};
\node[vertices] (w2) at (1.5,-3.3) {};

\node[vertices] (x1) at (0,0) {};
\node[vertices] (x2) at (.8,-.5) {};
\node[vertices] (x3)  at (2.2,-.5) {};
\node[vertices] (x4)  at (3,0){};
\node[vertices] (x5)  at (1.4,.3){};

\node[vertices] (y1) at (0,0-1) {};
\node[vertices] (y2) at (.8,-.5-1) {};
\node[vertices] (y3)  at (2.2,-.5-1) {};
\node[vertices] (y4)  at (3,0-1){};
\node[vertices] (y5)  at (1.4,.3-1){};

\node[vertices] (z1) at (0,0-2) {};
\node[vertices] (z2) at (.8,-.5-2) {};
\node[vertices] (z3)  at (2.2,-.5-2) {};
\node[vertices] (z4)  at (3,0-2){};
\node[vertices] (z5)  at (1.4,.3-2){};

\foreach \from/\to in {x1/x2,x2/x3,x3/x4,x1/y1,x2/y2,x3/y3,x4/y4,x1/w1,x2/w1,x3/w1,x4/w1,z1/w2,z2/w2,z3/w2,z4/w2,y1/x2,y2/x3,y3/x4,z1/z2,z2/z3,z3/z4,y1/z1,y2/z2,y3/z3,y4/z4,z1/y2,z2/y3,z3/y4}
	\draw[-] (\from)--(\to);
\foreach \from/\to in {y1/y2,y2/y3,y3/y4}
	\draw[line width=1.5pt] (\from)--(\to);
\foreach \from/\to in {x1/x5,x4/x5,x5/y5,x5/w1,z5/w2,y4/x5,y5/x1,z1/z5,z4/z5,y5/z5,z4/y5,z5/y1}
	\draw[dotted] (\from)--(\to);
\foreach \from/\to in {y1/y5,y4/y5}
	\draw[dotted, line width=1.5pt] (\from)--(\to);
\foreach \from/\to in {w1/w2}
	\draw[black!80,dashed] (\from)--(\to);
\end{tikzpicture}
\end{center}
If the edges in bold are chosen as the minimal generator of first homology, there is no way of partitioning the remaining vertices into hemispheres such that the upper and lower hemisphere intersect in this cycle, and the union of the induced complexes on the hemispheres is the entire complex.  The edge running from the top vertex to the bottom vertex will not be contained in the induced subcomplexes.
\end{example}
In the case when the Stanley-Reisner complex is a higher dimensional triangulation of a sphere, however, this theorem can be extended.
\begin{theorem}\label{polytopes}  Let $\Delta$ be a convex, clique-closed simplicial n-polytope which is a topological (n-1)-dimensional sphere.  If all minimal induced subcomplexes such that $\dim\homol{i}(\Delta')\neq 0$ are of the same size, then for $I_G$ an edge ideal with $\Delta_G$ of this form, the relative jump sequence of $I_G$ satisfies the inequalities
$$1\leq r_1\leq r_2\leq\cdots\leq r_{n-1}.$$
\end{theorem}
\begin{proof}  By construction, we have that $\Delta_G$ is a minimal generator of $H_{n-1}$, on some vertex set V.  Select a set of vertices $V_1$ of minimal size on which $\dim H_2(\Delta_G|_{V_1})\neq 0$, and a minimal set of vertices $V_2\subset V_1$ with $\dim H_{n-3}(\Delta_G|_{V_2})\neq 0$.  We view $\Delta_G|_{V_2}$ as an equator of $\Delta_G|_{V_1}$, and partition the vertices of $V_1$ into the upper hemisphere with boundary and the lower hemisphere with boundary, $V_{1,u}$ and $V_{1,l}$, which by construction gives $V_{1,u}\cap V_{1,l}=V_2$.\\
\\
As before, we form vertex sets $K=V\backslash V_1$, $K_u=K\cup V_{1,u}$, $K_l=K\cup V_{1,l}$, and $K'=K\cup V_2=K_u\cap K_l$.  Using the fact that $\Delta$ is a convex simplicial polytope, we have $\Delta=\left(K_u\right)\cup\left(K_l\right)$ and as before, $\Delta|_{K'}=\Delta|_{K_u}\cap \Delta|_{K_l}.$
The relevant part of the Mayer-Vietoris sequence is
\begin{equation*}
\cdots\rightarrow\homol{n-2}(\Delta_{K_u})\oplus \homol{n-2}(\Delta_{K_l})\rightarrow \homol{n-2}(\Delta_G)\xrightarrow{\partial} \homol{n-3}(\Delta_{K'})\rightarrow \cdots
\end{equation*}
which again gives us that $\homol{n-3}(\Delta|_{K'})\neq 0$.  Note that the number of vertices involved in the assumed minimal generator of $\homol{n-2}$, $V_2$, is $a_{n-2}+n=r_1+r_2+\cdots+r_{n-2}+n$, and the number of vertices involved in $\Delta_G$ and $V_1$ are respectively $a_n+(n+2)=a_{n-2}+r_{n-1}+r_{n-2}+(n+2)$ and $a_{n-1}+(n+1)=a_{n-2}+r_{n-1}+(n+1).$ To summarize,
\begin{align*}
|V|&=a_{n-2}+r_{n-1}+r_{n-2}+(n+2),\\
|V_1|&=a_{n-2}+r_{n-1}+(n+1),\\
|K|&=|V|-|V_1|=r_{n-2}+1,\\
|V_2|&=a_{n-2}+n,\;\;\text{and}\\
|K'|&=|K|+|V_2|=a_{n-2}+r_{n-2}+(n+1).
\end{align*}
Noting that $\homol{n-3}(\Delta|_{K'})\neq 0$, we combine this with the fact that $V_1$ is a set with supposed minimal size to obtain
$$|V_1|=a_{n-2}+r_{n-1}+(n+1)\leq a_{n-2}+r_{n-2}+(n+1),$$
giving us that $r_{n-1}\leq r_{n-2}$.\\
\\
It should be noted that the hypothesis that all minimal subsets are of the same size is a nontrivial assumption -- it is not necessarily the case that $V_2$ is of minimal size among all such subsets, and without that the proof fails.\\
\\
To complete the chain of inequalities, we iteratively repeat this argument for a minimal induced subset with nonzero $\homol{j-2}$ to calculate the remaining $r_j$, $j<n-2$.
\end{proof}




The bound obtained in Theorem \ref{maintheorem} can be improved directly for low $a_1$ via enumerative arguments.  In these cases, we conjecture that $a_1=2$, $a_2\geq 8$.  This bound, if true, is sharp by Example \ref{balloon}.  However, we have that if $a_1=2$, then $a_2\geq 7$.  The proof, to appear in \cite{Wh11}, is technical but we include a narrower result for a demonstration of these techniques.
\begin{theorem}\label{a2geq5}  Let $G$ be a graph such that $I_G$ has jump sequence $[k;2,a_2,...,a_n]$.  Then $a_2\geq 5$.
\end{theorem}
\begin{remark}  We note that in any graph which did satisfy $a_1=2$ and $a_2=4$, there would exist some (minimal) induced subcomplex of $\Gc$ on vertex set of size 8, $W\subseteq V$, such that $\dim\widetilde{H}_2(\Gc|_W)\neq 0$.  We first characterize properties such a complex would satisfy, then return to the main proof of Theorem \ref{a2geq5}.
\end{remark}
First, we show that $G^c$ contains no vertices of $\deg(v)=6,7$ when $G$ as above with $\dim\widetilde{H}_2(\Delta_G)\neq 0$.
\begin{lemma}\label{sixseven}  Let $G$ be a graph with no induced 4-cycles in the complement graph $G^c$, i.e. $G^c$ is $C_4$-free.  Then there exist no subsets of vertices $W\subseteq V$, $|W|=8$, such that $\dim\widetilde{H}_2(\Delta_G|W,k)\neq 0$ and with any $v\in W$ of $\deg(v)=6,7$.
\end{lemma}
\begin{proof}[Proof of Lemma \ref{sixseven}]
Assume that $W$ is any subset of $G$ with $|W|=8$ and
$$\dim\homol{2}(\Delta_G|_W)\neq 0.$$  Without loss of generality, we may assume that $G$ was a graph on vertex set W.  Choose the vertex of highest degree in the $G^c$.  If $\deg(v)=7$, then the vertex set of $G^c$ is equal to $\{v\}\cup n(v)$.  So $G^c$ is a cone over $(G^c)\backslash v$, and hence, is contractible and has no homology.  So $\deg(v)<7$ for all $v\in G$.\\
\\
If the highest degree vertex $v\in G^c$ is of degree six, then all but one of the vertices are adjacent to $v$.  We divide the proof up into cases by degree of this vertex $w\notin n(v)$.\\
\\
If $\deg(w)=1$, then we can remove $w$ to obtain a complex with the same second homology,
$$\dim\widetilde{H}_2(\Gc)=\dim\widetilde{H}_2(\Gc\backslash w).$$
However, as $|G^c\backslash w|=7$, this contradicts Theorem \ref{maintheorem}.\\
\\
Let $\deg(w)=r$ and $W=\{v_1,...,v_r\}=n(w)\subseteq n(v)$.  For each pair $v_i,v_j\in n(w)$, we must have $\{v_i,v_j\}\in G^c$.  Otherwise, we have an induced 4-cycle on $\{v,v_i,w,v_j\}$.\\
\\
So $\Gc$ is the clique closure of $\Susp(K_{r},\{v,w\})\cup\Cone(\widehat{G^c \backslash W})$, or the clique closure of the union (with appropriate identifications) of the suspension of the complete graph on $r$ vertices and a cone with $v$ over the remaining vertices $V\backslash W$.  This complex is contractible, and so again, we obtain a contradiction.  So no vertices of degrees 6 or 7 can be in the 1-skeleton of such a minimal $\widetilde{H}_2$ generator.
\end{proof}
\begin{lemma}\label{wheels}  Let $G$ be a graph and $\Delta=\Delta_G$ its Stanley-Reisner complex.  Let $\dhomol{2}(\Delta_G)\neq 0$ and all induced subcomplexes $\Delta'=\Delta_G|_W$ satisfy $\dhomol{2}(\Delta')=0$.  For each vertex $v\in G$, we have that some subset of the neighbors of v is an induced cycle of length $\geq 4$, i.e. some $S\subseteq n(v)$ with $\Delta|_S\cong C_{|S|}$.
\end{lemma}
\begin{proof}  Let $W=\{v\}\cup n(v)$ and $X=\Delta\backslash\{v\}$.  We note that $W\cap X=n(v)$, the neighbors of $v$.  As $\Delta|X\cup\Delta|W=\Delta$ and $\Delta|_W\cap\Delta|_X=\Delta|_{n(V)}$, we have another Meyer-Vietoris sequence,
\begin{equation*}
\cdots\rightarrow \homol{2}(\Delta|_W)\oplus \homol{2}(\Delta|_X)\rightarrow \homol{2}(\Delta)\xrightarrow{\partial} \homol{1}(\Delta|_{n(v)})\rightarrow \homol{1}(\Delta|_W)\oplus \homol{1}(\Delta|_X)\rightarrow\cdots.
\end{equation*}
For this, as the leftmost term is again zero, we have that $\homol{1}(\Delta|_{n(v)})\neq 0$ by injectivity of $\partial$.  So there must be some subset $S\subseteq n(v)$ for which we have an induced cycle generating nonzero first homology.  It must be the case that $n\geq 4$, as $\Delta_G$ is clique closed and hence a cycle of length 3 is filled in with a 2-simplex.
\end{proof}
\begin{lemma}\label{34vert}  Let $G$ be a graph on $8$ vertices with $C_4$-free complement, with
$$\dim\homol{2}(\Gc)\neq 0.$$
Then $G^c$ has no vertices of degrees 1,2, 3 or 4.
\end{lemma}
\begin{proof}[Proof of Lemma \ref{34vert}.]  We consider the first three cases together, and the last case separately:\\
\emph{\bf Case 1:}  \emph{(G contains a vertex of degree 1,2, or 3.)}  Let $v$ be a vertex of degree 1,2, or 3.  This directly contradicts Lemma \ref{wheels}, as there can be no induced cycle of length greater than or equal to 4 in the induced graph on the set of neighbors of $v$.\\
\\
\emph{\bf Case 2:}  \emph{(G contains a vertex of degree 4.)}  Let $v$ be a vertex of degree 4.  By Lemma \ref{wheels}, the only possibility is that $\Delta_G|_{n(v)}\cong C_4$.  However, this violates $C_4$-freeness of $G^c$.
\end{proof}
From Lemmas \ref{sixseven} and \ref{34vert}, we know that all graphs $G$ on eight vertices with a $C_4$ free complement cannot have any vertices of degrees 1-4, or degrees 6,7 in $G^c$.  Hence, $G^c$ must be a 5-regular graph.
\begin{proof}[Proof of Theorem \ref{a2geq5}]
By Lemma \ref{reductionlemma}, is sufficient to restrict to the case where $\reg(I_G)=4$ and G is taken to be a graph on vertex set $W$.  As noted in the Lemma, this $I_G$ may have $a_1>2$.  This gives that restricting to this subgraph may remove all of our induced 5-cycles, pushing $a_1$ from 2 up to 3 or higher.  However, from Theorem \ref{maintheorem}, this guarantees $a_2\geq 6.$
We restrict then to the case that $G$ is a graph on 8 vertices with $\dhomol{2}(\Delta_G)\neq 0$.  We note that this $G^c$ still must have no induced 4-cycles.\\
\\
By Lemmas \ref{sixseven} and \ref{34vert}, we have that such a graph must have a 5-regular complement graph, $G^c$.  Choosing any vertex $v$, the neighbors of $v$, $n(v)$, must have an induced cycle of length greater than or equal to 4 in $G^c$ - and by assumption, as $G^c$ is $C_4$-free, $G^c$ restricted to $V'=\{v\}\cup n(v)$ must be a 5-wheel.
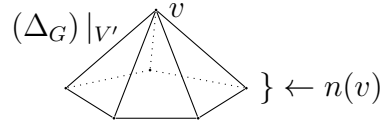
\begin{center}
\begin{tikzpicture}[scale=0.8, vertices/.style={circle, fill=black, inner sep=0pt}]

\node (D) at (0,1){$\left(\Delta_G\right)|_{V'}$};
\node[vertices, label=right:{$v$}] (w1) at (1.5,1.3) {};

\node[vertices] (x1) at (0,0) {};
\node[vertices] (x2) at (.8,-.5) {};
\node[vertices] (x3)  at (2.2,-.5) {};
\node[vertices, label=right:{$\}\leftarrow n(v)$}] (x4)  at (3,0){};
\node[vertices] (x5)  at (1.4,.3){};

\foreach \from/\to in {x1/x2,x2/x3,x3/x4,x1/w1,x2/w1,x3/w1,x4/w1}
	\draw[-] (\from)--(\to);
\foreach \from/\to in {x1/x5,x4/x5,x5/w1}
	\draw[dotted] (\from)--(\to);

\end{tikzpicture}
\end{center}
We consider the complex without the top vertex $v$.  There are two vertices $w_1$ and $w_2$ in $\Delta$ not contained in $n(v)$, each of degree 5.
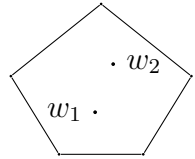
\begin{center}
\begin{tikzpicture}[scale=0.8, vertices/.style={circle, fill=black, inner sep=0pt}]

\node[vertices] (x1) at (0,0.3) {};
\node[vertices] (x2) at (.8,-1) {};
\node[vertices] (x3)  at (2.2,-1) {};
\node[vertices] (x4)  at (3,0.3){};
\node[vertices] (x5)  at (1.5,1.5){};

\node[vertices, inner sep=.5pt, label=left:{$w_1$}] (w1)  at (1.4,-.3){};
\node[vertices, inner sep=.5pt, label=right:{$w_2$}] (w2)  at (1.7,0.5){};

\foreach \from/\to in {x1/x2,x2/x3,x3/x4,x4/x5,x1/x5}
	\draw[-] (\from)--(\to);

\end{tikzpicture}
\end{center}
We can split the graphs of this form into two cases - either $w_1$ and $w_2$ share an edge in $\Delta$ or not.\\
\\
\emph{\bf Case 1:}  \emph{Vertices $w_1$ and $w_2$ do not lie on an edge.}  As $G^c$ is a 5-regular graph, we have that $w_1$ and $w_2$ must be adjacent to 5 vertices.  As they are not adjacent to one another or to w, it must be the case that all three form a cone over the vertices in $n(v)$.
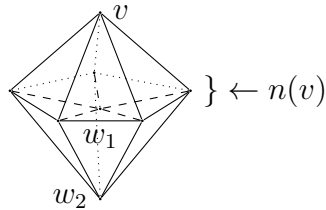
\begin{center}
\begin{tikzpicture}[scale=0.8, vertices/.style={circle, fill=black, inner sep=0pt}]

\node[vertices, label=right:{$v$}] (v) at (1.5,1.3) {};

\node[vertices] (x1) at (0,0) {};
\node[vertices] (x2) at (.8,-.5) {};
\node[vertices] (x3)  at (2.2,-.5) {};
\node[vertices, label=right:{$\}\leftarrow n(v)$}] (x4)  at (3,0){};
\node[vertices] (x5)  at (1.4,.3){};

\node[vertices] (w1) at (1.5,-.3) {};
\node (w1lab) at (1.5,-.8){$w_1$};
\node[vertices, label=left:{$w_2$}] (w2) at (1.5,-1.8) {};

\foreach \from/\to in {x1/x2,x2/x3,x3/x4,x1/v,x2/v,x3/v,x4/v}
	\draw[-] (\from)--(\to);
\foreach \from/\to in {x1/x5,x4/x5,x5/v,x5/w2}
	\draw[dotted] (\from)--(\to);

\foreach \from/\to in {x1/w2,x2/w2,x3/w2,x4/w2}
	\draw[-] (\from)--(\to);
\foreach \from/\to in {x5/w1,x1/w1,x2/w1,x3/w1,x4/w1}
	\draw[dashed] (\from)--(\to);

\end{tikzpicture}
\end{center}
This complex has numerous induced 4-cycles.  For example, $\{w_1,x_1,v,x_3\}$ will be one such.  So this violates our assumption that $G^c$ was 4-cycle free.\\
\\
\emph{\bf Case 2:}  \emph{Vertices $w_1$ and $w_2$ lie on an edge.}  Note that the other four vertices connected to $w_1$ and $w_2$ form two paths $P_1$ and $P_2$ of length 4 on the induced cycle of length 5 in $n(v)$.  The three non-isomorphic cases are:
\begin{enumerate}
\item $P_1=P_2$
\item $P_1\cap P_2$ is a path of length 3, i.e. the vertex in $n(v)$ not contained in $P_1$ and the vertex in $n(v)$ not contained in $P_2$ are adjacent, or
\item $P_1\cap P_2$ is disconnected, i.e. the vertex in $n(v)$ not contained in $P_1$ and the vertex in $n(v)$ not contained in $P_2$ are nonadjacent.
\end{enumerate}
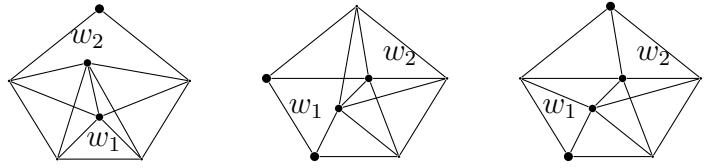
\begin{center}
\begin{tikzpicture}[scale=0.8, vertices/.style={circle, fill=black, inner sep=0pt}]

\node[vertices] (x1) at (0,0.3) {};
\node[vertices] (x2) at (.8,-1) {};
\node[vertices] (x3)  at (2.2,-1) {};
\node[vertices] (x4)  at (3,0.3){};
\node[vertices, inner sep=1.2pt] (x5)  at (1.5,1.5){};

\node[vertices, inner sep=1pt, label=below:{$\;w_1$}] (w1)  at (1.5,-.3){};
\node[vertices, inner sep=1pt, label=above:{$w_2$}] (w2)  at (1.3,0.6){};

\foreach \from/\to in {x1/w1,x2/w1,x3/w1,x4/w1,x1/w2,x2/w2,x3/w2,x4/w2,w1/w2}
	\draw[-] (\from)--(\to);

\foreach \from/\to in {x1/x2,x2/x3,x3/x4,x4/x5,x1/x5}
	\draw[-] (\from)--(\to);

\end{tikzpicture}\;\;\;\;\;\;\;
\begin{tikzpicture}[scale=0.8, vertices/.style={circle, fill=black, inner sep=0pt}]

\node[vertices, inner sep=1.2pt] (x1) at (0,0.3) {};
\node[vertices, inner sep=1.2pt] (x2) at (.8,-1) {};
\node[vertices] (x3)  at (2.2,-1) {};
\node[vertices] (x4)  at (3,0.3){};
\node[vertices] (x5)  at (1.5,1.5){};

\node[vertices, inner sep=1pt, label=left:{$w_1$}] (w1)  at (1.2,-.2){};
\node[vertices, inner sep=1pt, label=above right:{$w_2$}] (w2)  at (1.7,0.3){};

\foreach \from/\to in {x1/x2,x2/x3,x3/x4,x4/x5,x1/x5,w1/w2}
	\draw[-] (\from)--(\to);

\foreach \from/\to in {x1/w2,x3/w2,x4/w2,x5/w2,x2/w1,x3/w1,x4/w1,x5/w1}
	\draw[-] (\from)--(\to);

\end{tikzpicture}\;\;\;\;\;\;\;
\begin{tikzpicture}[scale=0.8, vertices/.style={circle, fill=black, inner sep=0pt}]

\node[vertices] (x1) at (0,0.3) {};
\node[vertices, inner sep=1.2pt] (x2) at (.8,-1) {};
\node[vertices] (x3)  at (2.2,-1) {};
\node[vertices] (x4)  at (3,0.3){};
\node[vertices, inner sep=1.2pt] (x5)  at (1.5,1.5){};

\node[vertices, inner sep=1pt, label=left:{$w_1$}] (w1)  at (1.2,-.2){};
\node[vertices, inner sep=1pt, label=above right:{$w_2$}] (w2)  at (1.7,0.3){};

\foreach \from/\to in {x1/x2,x2/x3,x3/x4,x4/x5,x1/x5,w1/w2,w2/x1,w2/x5,w2/x4,w2/x3,w1/x1,w1/x4,w1/x3,w1/x2}
	\draw[-] (\from)--(\to);

\end{tikzpicture}
\end{center}
In cases 2.1 and 2.2, the vertices in $n(v)$ which are not adjacent to both $w_1$ and $w_2$ are of degree 4 (including the edge running to $v$), with no other possible edges which can be added without increasing the degree of some vertex to six.  As the 1-skeleton of $\Delta_G$ is a 5-regular graph, this eliminates these cases.\\
\\
We reproduce and label the graph of case 2.3 on the left.  Note that the left graph has the vertices labeled $a$ and $c$ of degree 4.  As $G^c$ must be a 5-regular graph, and the only edge possible edge that can be added is the edge from $a$ to $c$, we must have the 1-skeleton of $\Delta_G$ is the graph on the right with the new dashed edge.
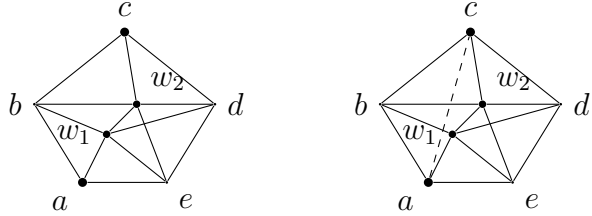
\begin{center}
\begin{tikzpicture}[scale=0.8, vertices/.style={circle, fill=black, inner sep=0pt}]

\node[vertices, label=left:{$b$}] (x1) at (0,0.3) {};
\node[vertices, inner sep=1.2pt, label=below left:{$a$}] (x2) at (.8,-1) {};
\node[vertices, label=below right:{$e$}] (x3)  at (2.2,-1) {};
\node[vertices, label=right:{$d$}] (x4)  at (3,0.3){};
\node[vertices, inner sep=1.2pt, label=above:{$c$}] (x5)  at (1.5,1.5){};

\node[vertices, inner sep=1pt, label=left:{$w_1$}] (w1)  at (1.2,-.2){};
\node[vertices, inner sep=1pt, label=above right:{$w_2$}] (w2)  at (1.7,0.3){};

\foreach \from/\to in {x1/x2,x2/x3,x3/x4,x4/x5,x1/x5,w1/w2,w2/x1,w2/x5,w2/x4,w2/x3,w1/x1,w1/x4,w1/x3,w1/x2}
	\draw[-] (\from)--(\to);

\end{tikzpicture}
\;\;\;\;\;\;\;\;
\begin{tikzpicture}[scale=0.8, vertices/.style={circle, fill=black, inner sep=0pt}]

\node[vertices, label=left:{$b$}] (x1) at (0,0.3) {};
\node[vertices, inner sep=1.2pt, label=below left:{$a$}] (x2) at (.8,-1) {};
\node[vertices, label=below right:{$e$}] (x3)  at (2.2,-1) {};
\node[vertices, label=right:{$d$}] (x4)  at (3,0.3){};
\node[vertices, inner sep=1.2pt, label=above:{$c$}] (x5)  at (1.5,1.5){};

\node[vertices, inner sep=1pt, label=left:{$w_1$}] (w1)  at (1.2,-.2){};
\node[vertices, inner sep=1pt, label=above right:{$w_2$}] (w2)  at (1.7,0.3){};

\foreach \from/\to in {x1/x2,x2/x3,x3/x4,x4/x5,x1/x5,w1/w2,w2/x1,w2/x5,w2/x4,w2/x3,w1/x1,w1/x4,w1/x3,w1/x2}
	\draw[-] (\from)--(\to);
\draw[dashed] (x2)--(x5);

\end{tikzpicture}
\end{center}
This has an induced 4-cycle in the 1-skeleton, violating our assumption of $G^c$ being $C_4$-free.  Including the edges running from $n(v)$ to $v$, we have that the 1-skeleton of the Stanley-Reisner complex of $I_G$ as labeled above has vertex set $V(G^c)=\{v,w_1,w_2,a,b,c,d,e\}$ and edges
\begin{align*}
E(G^c)=&\{ab,ac,ae,av,aw_1,bc,bv,bw_1,bw_2,cd,cv,cw_2,\\
&de,dv,dw_1,dw_2,ev,ew_1,ew_2,w_1w_2\}.
\end{align*}
So the graph $G$ is the graph of the cycle on $8$ vertices, with edge set:
$$E(G)=\{aw_2,w_2v,vw_1,w_1c,ce,eb,bd,da\}.$$
The resolutions of the edge ideals of the cycle graph have been well-studied, and for general $C_n$,
$$\text{Ind}(C_n)=\left\lfloor \frac{n}{3}\right\rfloor.$$
So our graph $G$ above must have $\Ind{G}=2$.\\
\\
So no graph $G$ with a $C_4$-free complement and $\homol{2}(\widehat{G^c})\neq 0$ on 8 vertices exists, and all graphs with $a_1=2$ must have $a_2\geq 5$.
\end{proof}

This theorem shows that notable constraints exist on the type of sygyzies found in edge ideals of graphs with $C_4$-free complement.  No edge ideals exist with betti diagrams of the forms:
\begin{center}
\begin{tikzpicture} [betti/.style={
execute at begin cell=\node\bgroup,
execute at end cell=\egroup;,
execute at empty cell=\node{$\cdot$};
},vertices/.style={circle, fill=black, inner sep=0pt}] 
\matrix [betti] {
-&0&1&2&3&4&5&6&$\cdots$\\
total: &1&$\ast$&$\ast$&$\ast$&$\ast$&$\ast$&$\ast$&$\cdots$\\
0: &1&&&&&&&\\
1:&&$\ast$&$\ast$&$\circ$&$\circ$&$\circ$&$\circ$&$\cdots$\\
2:&&&&$\ast$&$\ast$&$\circ$&$\circ$&$\cdots$\\
3:&&&&&$\ast$&$\ast$&$\circ$&$\cdots$\\
};
\node [vertices] (1) at (-1,0){};
\node [vertices] (2) at (-1,-.6){};
\node [vertices] (3) at (-0.05,-.6){};
\node [vertices] (4) at (-0.05,-1.1){};
\node [vertices] (5) at (.45,-1.1){};
\node [vertices] (6) at (.45,-1.6){};
\node [vertices] (7) at (.7,-1.6){};

\foreach \from/\to in {1/2,2/3,3/4,4/5,5/6,6/7}
	\draw [-] (\from)--(\to);

\end{tikzpicture}\;\;\;\;\;
\begin{tikzpicture} [betti/.style={
execute at begin cell=\node\bgroup,
execute at end cell=\egroup;,
execute at empty cell=\node{$\cdot$};
},vertices/.style={circle, fill=black, inner sep=0pt}] 
\matrix [betti] {
-&0&1&2&3&4&5&6&$\cdots$\\
total: &1&$\ast$&$\ast$&$\ast$&$\ast$&$\ast$&$\ast$&$\cdots$\\
0: &1&&&&&&&\\
1:&&$\ast$&$\ast$&$\circ$&$\circ$&$\circ$&$\circ$&$\cdots$\\
2:&&&&$\ast$&$\ast$&$\circ$&$\circ$&$\cdots$\\
3:&&&&&&$\ast$&$\circ$&$\cdots$\\
};
\node [vertices] (1) at (-1,0){};
\node [vertices] (2) at (-1,-.6){};
\node [vertices] (3) at (-0.05,-.6){};
\node [vertices] (4) at (-0.05,-1.1){};
\node [vertices] (5) at (.9,-1.1){};
\node [vertices] (6) at (.9,-1.6){};
\node [vertices] (7) at (1.1,-1.6){};

\foreach \from/\to in {1/2,2/3,3/4,4/5,5/6,6/7}
	\draw [-] (\from)--(\to);
\end{tikzpicture}
\end{center}
In these Betti diagrams, the $\ast$ indicate necesssarily nonzero $\betti{i}{j}$ and the $\circ$ indicate possible $\betti{i}{j}$.  It remains open whether or not any edge ideals exist with a betti diagram with $\betti{2}{4}(I_G)=0$ as below and with nonzero Betti numbers between the two lines in the table below.  From \cite{Wh11}, it must be that if $\betti{2}{4}(I_G)=0$, then $\betti{3}{7}(I_G)=\betti{3}{8}(I_G)=\betti{3}{9}(I_G)=0$.
\begin{center}
\begin{tikzpicture} [betti/.style={
execute at begin cell=\node\bgroup,
execute at end cell=\egroup;,
execute at empty cell=\node{$\cdot$};
},vertices/.style={circle, fill=black, inner sep=0pt}] 
\matrix [betti] {
-&0&1&2&3&4&5&6&7&8&$\cdots$\\
total: &1&$\ast$&$\ast$&$\ast$&$\ast$&$\ast$&$\ast$&$\cdots$\\
0: &1&&&&&&&&&\\
1:&&$\ast$&$\ast$&$\circ$&$\circ$&$\circ$&$\circ$&$\circ$&$\circ$&$\cdots$\\
2:&&&$0$&$\ast$&$\ast$&$\ast$&$\ast$&$\ast$&$\ast$&$\circ$\\
3:&&&&&&&&&&$\ast$\\
};
\node [vertices] (1) at (-1.7,0){};
\node [vertices] (2) at (-1.7,-.6){};
\node [vertices] (3) at (-0.65,-.6){};
\node [vertices] (4) at (-0.65,-1.1){};
\node [vertices] (5) at (2.7,-1.1){};
\node [vertices] (6) at (2.7,-1.6){};
\node [vertices] (7) at (3.1,-1.6){};

\node [vertices] (5') at (1.3,-1.1){};
\node [vertices] (6') at (1.3,-1.6){};

\foreach \from/\to in {1/2,2/3,3/4,4/5,5/6,6/7}
	\draw [-] (\from)--(\to);

\foreach \from/\to in {5'/6',6'/6}
	\draw [dashed] (\from)--(\to);
\end{tikzpicture}
\end{center}

So we have a bound on the increase in degrees of the sygygies of $I_G$ in terms of minimal cycle length in $G^c$.  It is not however the case that for all graphs, these relative jump sequences (i.e. the lengths of the \emph{stairs} of the lower edge of the resolution) must be weakly increasing.  In Example \ref{nonjump} in Section \ref{cornerclasses} we construct such a counterexample.  We also provide a general algorithm for constructing large classes of Betti diagrams of edge ideals.  It should be noted that the number of vertices involved in the example are high - and no graphs of smaller size are currently known whose edge ideals exhibit this behavior.



\section{Corner Diagrams and Jump Sequences}\label{cornerclasses}

In this section, we describe a technique of producing jump sequences which are not of the form found in Theorem \ref{polytopes}.  Specifically, we construct a counterexample to all relative jump sequences being weakly increasing.  Throughout, we will refer to \emph{Betti diagrams of shape ${\bf a}=[k;a_1,...,a_{k-1}]$}, or ${\mathcal B}_{\bf a}$, the set of all betti diagrams of edge ideals with jump sequence {\bf a}.
\begin{definition}  The \emph{corner sum of two jump sequences} ${\bf a}=[k;a_1,a_2,...,a_{k-1}]$ and\\
${\bf b}=[l;b_1,b_1,...,b_{k-1},b_{j},...,b_{l-1}]$, where $k\leq l$, we define to be
$$[l;c_1,c_2,...,c_{k-1},b_{j},...,b_{l-1}],$$
with $c_i=\min\{a_i,b_i\}.$  We denote this corner sum ${\bf a}\oplus{\bf b}$.
\end{definition}
\begin{example}  The corner sum of two jump sequences can be thought of as the jump sequence obtained by superimposing the Betti diagrams of two edge ideals $I_G$ and $I_{G'}$ on top of one another.
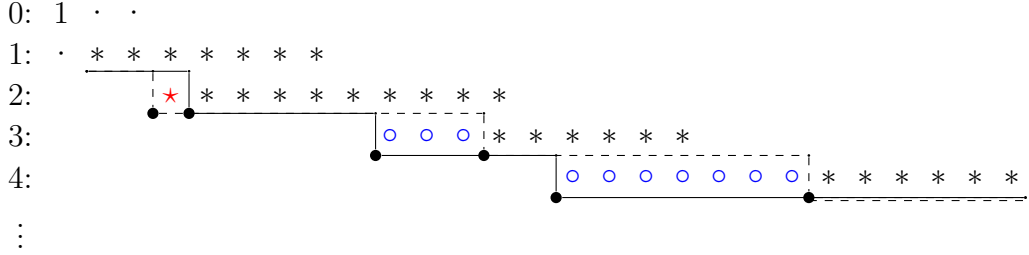
\begin{center}
\begin{tikzpicture} [scale=.8,auto=left,betti/.style={
execute at begin cell=\node\bgroup,
execute at end cell=\egroup;,
execute at empty cell=\node{$\cdot$};
},
vertices/.style={circle, fill=black, inner sep=0pt}]

\node [vertices] (r1) at (-5-2.1,.9){};
\node [vertices] (r2) at (-3.3-2.1,.9){};
\node [vertices, inner sep=1.5pt] (r3) at (-3.3-2.1,.2){};
\node [vertices] (r4) at (-0.2-2.1,.2){};
\node [vertices, inner sep=1.5pt] (r5) at (-0.2-2.1,-.5){};
\node [vertices] (r6) at (.7,-.5){};
\node [vertices, inner sep=1.5pt] (r7) at (.7,-1.2){};
\node [vertices] (r8) at (8.5,-1.2){};

\node [vertices] (s1) at (-5-2.1,.9){};
\node [vertices] (s2) at (-3.3-.6-2.1,.9){};
\node [vertices, inner sep=1.5pt] (s3) at (-3.3-.6-2.1,.2){};
\node [vertices] (s4) at (-0.2+1.8-2.1,.2){};
\node [vertices, inner sep=1.5pt] (s5) at (-0.2+1.8-2.1,-.5){};
\node [vertices] (s6) at (4.9,-.5){};
\node [vertices, inner sep=1.5pt] (s7) at (4.9,-1.2){};
\node [vertices] (s8) at (8.5,-1.2){};
\draw [dashed] (4.9,-1.2-.05)--(8.5,-1.2-.05);

\foreach \from/\to in {r1/r2, r2/r3,r3/r4,r4/r5,r5/r6,r6/r7,r7/r8}
	\draw[-] (\from)--(\to);

\foreach \from/\to in {s1/s2, s2/s3,s3/s4,s4/s5,s5/s6,s6/s7,s7/s8}
	\draw[dashed] (\from)--(\to);

\matrix [betti,row 3 column 5/.style=red, row 4 column 11/.style=blue,row 4 column 12/.style=blue,row 4 column 13/.style=blue,row 5 column 17/.style=blue,row 5 column 18/.style=blue,row 5 column 19/.style=blue,row 5 column 20/.style=blue,row 5 column 21/.style=blue,row 5 column 22/.style=blue,row 5 column 16/.style=blue] {
0: &1&&\\
1:& &$\ast$&$\ast$&$\ast$&$\ast$&$\ast$&$\ast$&$\ast$\\
2:&\;&\;&\;&$\star$&$\ast$&$\ast$&$\ast$&$\ast$&$\ast$&$\ast$&$\ast$&$\ast$&$\ast$\\
3:&\;&\;&\;&\;&\;&\;&\;&\;&\;&$\circ$&$\circ$&$\circ$&$\ast$&$\ast$&$\ast$&$\ast$&$\ast$&$\ast$\\
4:&\;&\;&\;&\;&\;&\;&\;&\;&\;&\;&\;&\;&\;&\;&$\circ$&$\circ$&$\circ$&$\circ$&$\circ$&$\circ$&$\circ$&$\ast$&$\ast$&$\ast$&$\ast$&$\ast$&$\ast$\\
$\vdots$&\;&\;&\;&\;&\;&\;&\;&\;\\
};
\end{tikzpicture}
\end{center}
In this case, the Betti table of $I_G$ lies above the dashed line, and the Betti table of $I_G'$ lies above the solid line.  Betti numbers of solely $I_G$ are indicated by the $\star$ and Betti numbers of solely $I_{G'}$ are indicated by $\circ$. The jump sequence of $I_G$ is $[4;2,11,20]$ and the jump sequence of $I_{G'}$ is $[4;3,8,13]$, with the corner sum of these two jump sequences given by $[4;2,8,14]$.
\end{example}
We use this corner sum to describe possible jump sequences as follows:
\begin{proposition}\label{cornersumgraphs}  Given two Betti diagrams of edge ideals $I_G$, $I_H$, with jump sequences ${\bf a}$ and ${\bf b}$ respectively, we have a graph $K$ such that $I_K$ has jump sequence ${\bf a}\oplus{\bf b}$.
\end{proposition}
Given graphs $G$ and $H$ on vertex sets $\{v_1,...,v_n\}$ and $\{w_1,...,w_m\}$, we form graph $K$ on vertex set $\{v_1,...,v_n,w_1,...,w_m\}$, with edge sets:
\begin{enumerate} 
\item $\{e: e\in E_G\}$
\item $\{e: e\in E_H\}$
\item $\{\{v_i,w_j\}:1\leq i\leq n, 1\leq j\leq m\}.$
\end{enumerate}
This graph has Stanley-Reisner complex $\Delta_K=\Delta_G\dotcup \Delta_H$, the disjoint union of the Stanley-Reisner complexes of $\Delta_G$ and $\Delta_H$.  We note that the Betti numbers of this complex can be computed via Hochster's formula in terms of sums of the Betti numbers of our original complexes as follows:\\
\\
\begin{lemma}\label{sumformulas} Let $G$, $H$, and $K$ be graphs as above.  Then for terms in the linear strand, we have
\begin{align*}
\betti{i}{i+1}(I_K)&=\betti{i}{i+1}(I_G)+\betti{i}{i+1}(I_H)\\
&+\sum_{j=1}^i\left(\binom{m}{i-j+1}\betti{j-1}{j}(I_G)+\binom{n}{j}\betti{i-j}{i-j+1}(I_H)\right)\\
&+\binom{m+n}{i+1}-\binom{m}{i+1}-\binom{n}{i+1}.
\end{align*}
For terms in the nonlinear strands, we have for $s\geq 2$,
\begin{align*}
\betti{i}{i+s}(I_K)&=\betti{i}{i+s}(I_G)+\betti{i}{i+s}(I_H)\\
&+\sum_{j=1}^{i+s-1}\left(\binom{m}{i-j+s}\betti{j-s}{j}(I_G)+\binom{n}{j}\betti{i-j}{i-j+s}(I_H)\right).\\
\end{align*}
\end{lemma}
\begin{proof} [Proof of Lemma \ref{sumformulas}]  Using Hochster's formula, we rewrite $\betti{i}{i+s}(I_K)$ in terms of the dimensions of the homologies of sets of size $i+s$.  For terms in the linear strand (for which s=1) this becomes:
\begin{align*}  \betti{i}{i+1}(I_K)&=\sum_{|W|=i+1}\homol{0}\left(\Delta_K|_W\right)\\
&=\sum_{\substack{
|W|=i+1\\
W\subseteq V_1}
}\homol{0}\left(\Delta_G|_W\right)+\sum_{\substack{
|W|=i+1\\
W\subseteq V_2}
}\homol{0}\left(\Delta_H|_W\right)\\
&\;\;\;\;\;+\sum_{\substack{
|R|+|S|=i+1\\
R\subseteq V_1,\;\;|R|=j\\
S\subseteq V_2,\;\;|S|=i-j+1}
}\left[\homol{0}\left(\Delta_G|_R\right)+\homol{0}\left(\Delta_H|_S\right)+1\right]
\end{align*}

The extra 1 in the rightmost summand corrects the count for reduced homology of the two subsets.  The first two terms in the summand are the Betti numbers of the original ideals.  We rewrite the sum using this, with $R\subseteq V_1$ and $S\subseteq V_2$, then sum across all subsets with the appropriate counts:

\begin{align*}  \betti{i}{i+1}(I_K)&=\betti{i}{i+1}(I_G)+\betti{i}{i+1}(I_H)\\
&\;\;\;\;\;+\sum_{j=1}^i\sum_{|S|=i-j+1}\left(\sum_{|R|=j}\left[\homol{0}\left(\Delta_G|_R\right)+\homol{0}\left(\Delta_H|_S\right)+1\right]\right)\\
&=\betti{i}{i+1}(I_G)+\betti{i}{i+1}(I_H)\\
&\;\;\;\;\;+\sum_{j=1}^i\sum_{|S|=i-j+1}\left(\betti{j-1}{j}(I_G)+\binom{n}{j}\homol{0}\left(\Delta_H|_S\right)+\binom{n}{j}\right)\\
&=\betti{i}{i+1}(I_G)+\betti{i}{i+1}(I_H)\\
&\;\;\;\;\;+\sum_{j=1}^i\sum_{|S|=i-j+1}\left(\betti{j-1}{j}(I_G)+\binom{n}{j}\homol{0}\left(\Delta_H|_S\right)+\binom{n}{j}\right)\\
&=\betti{i}{i+1}(I_G)+\betti{i}{i+1}(I_H)\\
&\;\;\;\;\;+\sum_{j=1}^i\left(\binom{m}{i-j-1}\betti{j-1}{j}(I_G)+\binom{n}{j}\betti{i-j}{i-j+1}\right)\\
&\;\;\;\;\;+\sum_{j=1}^i\left(\binom{m}{i-j+1}\binom{n}{j}\right)\\
\end{align*}

The final Betti number count above uses the combinatorial identity

\begin{align*}
\sum_{j=1}^i\binom{m}{i-j+1}\binom{n}{j}&=\left[\sum_{j=0}^{i+1}\binom{m}{i-j+1}\binom{n}{j}\right]-\binom{m}{i+1}-\binom{n}{i+1}\\
&=\binom{m+n}{i+1}-\binom{m}{i+1}-\binom{n}{i+1}
\end{align*}
This finishes the proof for the calculation of Betti numbers in the linear strand, producing the formula above.\\
\\
The count for the Betti numbers $\betti{i}{i+s}$ in the nonlinear strands is identical, removing the binomial coefficient terms coming from the reduced homology zero correction.
\end{proof}

\begin{proof}[Proof of Proposition \ref{cornersumgraphs}]  Let $[k;a_1,...,a_{k-1}]$ and $[l;b_1,...,b_{l-1}]$ be the jump sequences of $I_G$ and $I_H$ respectively.  From Proposition \ref{sumformulas}, we can see that $\betti{i}{i+s}(I_K)$ will be nonzero precisely when one or the other of $\betti{k}{k+s}(I_G)$ or $\betti{k}{k+s}(I_K)$ is nonzero, for some $k\leq i$.\\
\\
In terms of the Betti numbers on the right edge of the Betti table, we see then that the minimal nonzero betti number in each row should be in position $c_i$, where
$$c_i=\min\{a_i,b_i\}.$$
This completes our proof, and we can see that the lower edge of the Betti table of $I_K$ is obtained by superimposing the lower edges of the Betti tables of $I_G$ and $I_H$.  This gives us an edge ideal with jump sequence $[l;c_1,c_2,...,c_{k-1},b_k,...,b_l]$ as described.
\end{proof}
As a result, the shapes of Betti diagrams of edge ideals can be imbued with a monoid structure.  In Example \ref{tori}, we use this to construct a relative jump sequence which is not weakly increasing.

\begin{proposition}\label{tori}  Let $4\leq n_1\leq n_2\leq\cdots\leq n_r$ be a set of integers, and form the graphs $G_1=C_{n_1}^c,G_2=C_{n_2}^c,...,G_r=C_{n_r}^c$ on vertex sets $V_i=\{v_{i,j}:1\leq j\leq n_i\}$ for $1\leq i\leq r$.  Then the graph on vertex set $V=\dotcup V_i$ of $G=\dotcup G_i$ has an edge ideal with $\reg(I_G)=2r+1$, with relative jump sequence
$${\bf r}=[2r;\overbrace{1,1,...,1}^{r-1},n_1-3,n_2-3,...,n_r-3].$$
\end{proposition}
\begin{proof}  This follows from a quick note on the description of $\Delta_G$ in terms of the $\Delta_{G_i}$.  As in general, the Stanley-Reisner complex $\Delta_I$ of a monomial ideal $I=J+K$, where $J$ and $K$ are monomial ideals on disjoint sets of variables, satisfies
$$\Delta_I\cong \Delta_J\ast \Delta_K,$$
i.e. $\Delta_I$ is the join of the two subcomplexes $\Delta_J$ and $\Delta_K$.  In particular, we have
$$\Delta_G=\Delta_{G_1}\ast\Delta_{G_2}\ast\cdots\ast\Delta_{G_r}.$$
As a result, our Betti diagram $B$ of the edge ideal of $G$ can be written as the products as matrices of the Betti diagrams $B_i$ of the $G_i$, with $B=B_1B_2\cdots B_r$.  The regularity count and the jump sequence calculation follow from an immediate linear algebra computation.
\end{proof}
We use edge ideals of this form, in conjunction with Proposition \ref{cornersumgraphs}, to construct an example of an edge ideal whose jump sequence is not weakly increasing.\\
\begin{proposition}\label{nonjump}  There exists a graph $G$ with relative jump sequence $[k;r_1,...,r_{k-1}]$ such that $r_i\geq r_{i+1}$ for some $i$.
\end{proposition}
\begin{proof}  Let $G_1=C_5^c\;\dotcup\;(C_5^c)'\;\dotcup\;(C_5^c)''$ and $G_2=C_4^c\;\dotcup \;C_6^c\;\dotcup\;(C_6^c)'$ be two graphs, with $G_1$ the union of three 5-anticycle graphs and $G_2$ the union of one 4-anticycle and two 6-anticycles, all viewed as graphs on disjoint sets of vertices.  Using Proposition \ref{tori}, we have the relative jump sequences of $G_1$ and $G_2$ are respectively $r_1=[6;1,1,2,2,2]$ and $r_2=[6;1,1,1,3,3]$.  This gives us jump sequences
$$a_1=[6;1,2,4,6,8]\;\;\;\;\text{and}\;\;\;\;a_2=[6;1,2,3,6,9].$$
So using Proposition \ref{cornersumgraphs}, we have a graph $G$ which has jump sequence
$$a=a_1\oplus a_2=[6;1,2,3,6,8],$$
which gives us a relative jump sequence $r=[6;1,1,1,3,2].$
\end{proof}
\begin{remark}  Another interesting question to ask is what additional necessary conditions are required for $G$ to guarantee an increasing relative jump sequences.  Alternately, it would be of combinatorial interest to find classes of complexes where these relative jump sequences are not weakly increasing, but with $\Delta$ connected [excluding trivial cases like coning over a vertex to connect these two tori, etc.]
\end{remark}




\section{Classes of Graphs with $\Ind{G}=1$ and High Regularity}\label{highregclasses}
\begin{theorem}  Fix $n\geq 5$.  Let $H$ be the graph on vertex set $\{x_1,...,x_n,y_1,...,y_n,z_1,z_2\}$, with edges of the following forms:

\begin{enumerate}
\item $\{x_iz_1: 1\leq i\leq n\}$
\item $\{y_iz_2: 1\leq i\leq n\}$
\item $\{x_iy_i: 1\leq i\leq n\}$
\item $\{x_iy_{i+1}: 1\leq i\leq n-1\}$
\item $\{y_1x_n\}$
\end{enumerate}
Then $G=H^c$ has a Gorenstein edge ideal $I_G$ and a shellable Stanley-Reisner complex $\Delta_G=\widehat{H^c}$.  This ideal has jump sequence $[3; 2,2n-2]$.
\end{theorem}
\begin{proof}  This is clear from an examination of the clique closure of $H$.
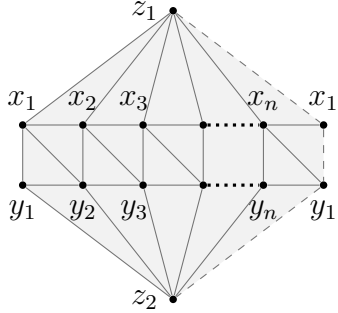
\begin{center}
\begin{tikzpicture}
[scale=.8,auto=left,vertices/.style={circle, fill=black, inner sep=1pt}]

\fill [black!05] (0,0)--(0,1)--(2.5,2.9)--(5,1)--(5,0)--(2.5,-1.9)--cycle;
\node [vertices, label=above:{$x_1$}] (x1) at (0,1){};
\node [vertices, label=above:{$x_2$}] (x2) at (1,1){};
\node [vertices, label=above:{$x_3\;\;$}] (x3) at (2,1){};
\node [vertices] (x4) at (3,1){};
\node [vertices, label=above:{$x_n$}] (xn) at (4,1){};
\node [vertices, label=above:{$x_1$}] (x1') at (5,1){};

\node [vertices,label=left:{$z_1$}] (z1) at (2.5,2.9){};

\node [vertices,label=below:{$y_1$}] (y1) at (0,0){};
\node [vertices,label=below:{$y_2$}] (y2) at (1,0){};
\node [vertices,label=below:{$y_3\;\;$}] (y3) at (2,0){};
\node [vertices] (y4) at (3,0){};
\node [vertices,label=below:{$y_n$}] (yn) at (4,0){};
\node [vertices,label=below:{$y_1$}] (y1') at (5,0){};

\node [vertices,label=left:{$z_2$}] (z2) at (2.5,-1.9){};

\foreach \from/\to in {z1/x1,z1/x2,z1/x3,z1/x4,z1/xn}
	\draw [black!55] (\from) -- (\to);

\foreach \from/\to in {z2/y1,z2/y2,z2/y3,z2/y4,z2/yn}
	\draw [black!55] (\from) -- (\to);
		
\foreach \from/\to in {x1/x2,x2/x3,x3/x4,xn/x1'}
	\draw [black!55] (\from) -- (\to);

\foreach \from/\to in {y1/y2,y2/y3,y3/y4,yn/y1'}
	\draw [black!55] (\from) -- (\to);

\foreach \from/\to in {x1/y1,x2/y2,x3/y3,x4/y4,xn/yn}
	\draw [black!55] (\from) -- (\to);
	
\foreach \from/\to in {x1/y2,x2/y3,x3/y4,xn/y1'}
	\draw [black!55] (\from) -- (\to);
	
\foreach \from/\to in {z1/x1',x1'/y1',y1'/z2}
	\draw [black!55,dashed] (\from) -- (\to);

\foreach \from/\to in {x4/xn,y4/yn}
	\draw [line width=1.2pt, dotted] (\from)--(\to);

\end{tikzpicture}
\end{center}
Unfolding the complex, we can see that it is homotopic to a 2-sphere, and by direct examination we note that it is both clique closed and 4-cycle free.  As these can be realized as convex triangulations of $S^2$, ideals of this form are Gorenstein.  The smallest induced cycles are of length 5 [for example, $\{x_1y_1,x_1x_2,x_2y_3,y_3z_2,y_1z_2\}$,] and there are 2n+2 total vertices in $G$, so we have the desired jump sequence.  In the case of $n=5$, we obtain Example \ref{balloon}, which was the icosahedron.
\end{proof}
All examples edge ideals with jump sequences $[3;a_1,a_2]$ considered so far have had $a_1=1$ or $a_1=2$.  We present a (non-Cohen-Macaulay) example of a complex with $a_1=3$ and $a_2=nm-4$ for any $n,m\geq 6$.
\begin{example}  Let $H$ be a graph on vertex set $\{x_{i,j}:1\leq i\leq n, 1\leq j\leq m\}$, with edges of the following forms:
\begin{enumerate}
\item $\{x_{i,j}x_{i,j+1}: 1\leq i\leq n, 1\leq j\leq m-1\}$
\item $\{x_{i,1}x_{i,n}, 1\leq i\leq n\}$
\item $\{x_{i,j}x_{i+1,j}: 1\leq i\leq n-1, 1\leq j\leq m\}$
\item $\{x_{1,j}x_{n,j}, 1\leq j\leq m\}$
\item $\{x_{i,j}x_{i+1,j+1}: 1\leq i\leq n, 1\leq j\leq m\}$
\item $\{x_{i,1}x_{i+1,m}: 1\leq i\leq n-1\}$
\item $\{x_{1,j}x_{n,j+1}: 1\leq j\leq m-1\}$
\item $\{x_{1,1}x_{n,n}\}$
\end{enumerate}
Then $G=H^c$ will have an edge ideal with jump sequence $[3;3,nm-4]$.\\
\\
This is clear from an examination of $H$, which forms the 1-skeleton of $\Delta_G$.
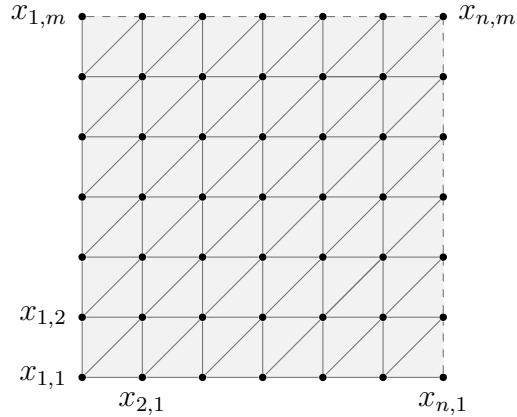
\begin{center}
\begin{tikzpicture}
[scale=.8,auto=left,vertices/.style={circle, fill=black, inner sep=1pt}]

\fill [black!05] (0,1)--(0,7)--(6,7)--(6,1)--cycle;
\node [vertices, label=left:{$x_{1,1}$}] (x1) at (0,1){};
\node [vertices, label=below:{$x_{2,1}$}] (x2) at (1,1){};
\node [vertices] (x3) at (2,1){};
\node [vertices] (x4) at (3,1){};
\node [vertices] (x5) at (4,1){};
\node [vertices] (xn) at (5,1){};
\node [vertices, label=below:{$x_{n,1}$}] (x1') at (6,1){};

\node [vertices, label=left:{$x_{1,2}$}] (y1) at (0,2){};
\node [vertices] (y2) at (1,2){};
\node [vertices] (y3) at (2,2){};
\node [vertices] (y4) at (3,2){};
\node [vertices] (y5) at (4,2){};
\node [vertices] (yn) at (5,2){};
\node [vertices] (y1') at (6,2){};

\node [vertices] (z1) at (0,3){};
\node [vertices] (z2) at (1,3){};
\node [vertices] (z3) at (2,3){};
\node [vertices] (z4) at (3,3){};
\node [vertices] (z5) at (4,3){};
\node [vertices] (zn) at (5,3){};
\node [vertices] (z1') at (6,3){};

\node [vertices] (w1) at (0,4){};
\node [vertices] (w2) at (1,4){};
\node [vertices] (w3) at (2,4){};
\node [vertices] (w4) at (3,4){};
\node [vertices] (w5) at (4,4){};
\node [vertices] (wn) at (5,4){};
\node [vertices] (w1') at (6,4){};

\node [vertices] (s1) at (0,5){};
\node [vertices] (s2) at (1,5){};
\node [vertices] (s3) at (2,5){};
\node [vertices] (s4) at (3,5){};
\node [vertices] (s5) at (4,5){};
\node [vertices] (sn) at (5,5){};
\node [vertices] (s1') at (6,5){};

\node [vertices] (t1) at (0,6){};
\node [vertices] (t2) at (1,6){};
\node [vertices] (t3) at (2,6){};
\node [vertices] (t4) at (3,6){};
\node [vertices] (t5) at (4,6){};
\node [vertices] (tn) at (5,6){};
\node [vertices] (t1') at (6,6){};

\node [vertices, label=left:{$x_{1,m}$}] (r1) at (0,7){};
\node [vertices] (r2) at (1,7){};
\node [vertices] (r3) at (2,7){};
\node [vertices] (r4) at (3,7){};
\node [vertices] (r5) at (4,7){};
\node [vertices] (rn) at (5,7){};
\node [vertices, label=right:{$x_{n,m}$}] (r1') at (6,7){};

\foreach \from/\to in {x1/x2,x2/x3,x3/x4,x4/x5,xn/x1'}
	\draw [black!55] (\from) -- (\to);

\foreach \from/\to in {y1/y2,y2/y3,y3/y4,y4/y5,yn/y1'}
	\draw [black!55] (\from) -- (\to);

\foreach \from/\to in {z1/z2,z2/z3,z3/z4,z4/z5,zn/z1'}
	\draw [black!55] (\from) -- (\to);

\foreach \from/\to in {w1/w2,w2/w3,w3/w4,w4/w5,wn/w1'}
	\draw [black!55] (\from) -- (\to);

\foreach \from/\to in {s1/s2,s2/s3,s3/s4,s4/s5,sn/s1'}
	\draw [black!55] (\from) -- (\to);

\foreach \from/\to in {t1/t2,t2/t3,t3/t4,t4/t5,t5/tn,tn/t1'}
	\draw [black!55] (\from) -- (\to);

\foreach \from/\to in {r1/r2,r2/r3,r3/r4,r4/r5,r5/rn,rn/r1'}
	\draw [dashed, black!55] (\from) -- (\to);

\foreach \from/\to in {x1/y1,x2/y2,x3/y3,x4/y4,x5/y5,xn/yn}
	\draw [black!55] (\from) -- (\to);
	
\foreach \from/\to in {z1/y1,z2/y2,z3/y3,z4/y4,z5/y5,zn/yn}
	\draw [black!55] (\from) -- (\to);

\foreach \from/\to in {z1/w1,z2/w2,z3/w3,z4/w4,z5/w5,zn/wn}
	\draw [black!55] (\from) -- (\to);

\foreach \from/\to in {s1/w1,s2/w2,s3/w3,s4/w4,s5/w5,sn/wn}
	\draw [black!55] (\from) -- (\to);

\foreach \from/\to in {s1/t1,s2/t2,s3/t3,s4/t4,s5/t5,sn/tn}
	\draw [black!55] (\from) -- (\to);

\foreach \from/\to in {r1/t1,r2/t2,r3/t3,r4/t4,r5/t5,rn/tn}
	\draw [black!55] (\from) -- (\to);

\foreach \from/\to in {x1/y2,x2/y3,x3/y4,x4/y5,xn/y1',z1/w2,z2/w3,z3/w4,z4/w5,zn/w1',w1/s2,w2/s3,w3/s4,w4/s5,wn/s1',s1/t2,s2/t3,s3/t4,s4/t5,sn/t1',y1/z2,y2/z3,y3/z4,y4/z5,y5/zn,yn/z1',t1/r2,t2/r3,t3/r4,t4/r5,t5/rn,tn/r1',x5/yn,y5/zn,z5/wn,w5/sn,s5/tn}
	\draw [black!55] (\from) -- (\to);
	
\foreach \from/\to in {x1'/y1',y1'/z1',z1'/w1',w1'/s1',s1'/t1',t1'/r1'}
	\draw [black!55,dashed] (\from) -- (\to);

\foreach \from/\to in {x5/xn,y5/yn,z5/zn,w5/wn,s5/sn,t5/tn}
	\draw [black!55] (\from)--(\to);

\end{tikzpicture}
\end{center}
The clique closure of this 1-skeleton is a torus on $nm$ vertices, with smallest induced cycles of length 6.  No induced proper subcomplex has nonzero second homology, so we must have jump sequence $[3;3,nm-4]$.
\end{example}
\begin{remark}  Each of these classes of graphs with $\Ind{G}=1$ and regularity 3,4, or 5 give rise to edge ideals with $\Ind{G'}=k$ and regularity $\reg(I_{G'})=2k+1$, $\reg(I_{G'})=3k+1$, and $\reg(I_{G'})=4k+1$, respectively.  Given a graph $G$ with regularity $r$, taking $G'$ to be $k$ disjoint copies of the graph on different sets of variables gives a Stanley-Reisner complex:
$$\Delta_{G'}=\Delta_G\ast \Delta_G\ast\cdots \Delta_G,$$
via combinatorial joins of the faces in $\Delta_G$.  Via the K\"unneth formula, we see we have nonzero homology in the desired degrees, giving us the desired regularity calculation.
\end{remark}
A more general way of constructing graphs with a $C_4$-free complement is desirable.  Given any triangulation of a 2-sphere $\Delta$, there is a way of retriangulating the sphere to produce a new complex $\sdd(\Delta)$ which is the Stanley-Reisner complex of an edge ideal $I_G$ with a $C_4$-free 1-skeleton.
\begin{definition}  Let $\Delta$ be a pure dimensional simplicial complex whose facets are all of dimension 2.  Then $\sdd(\Delta)$ is the simplicial complex obtained by replacing each facet with the following complex:
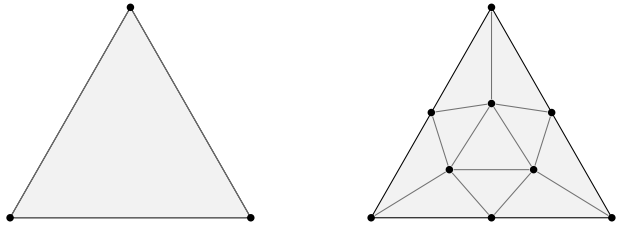
\begin{center}
\begin{tikzpicture}
[scale=.8,auto=left,vertices/.style={circle, fill=black, inner sep=1pt}]

\draw [fill=black!05] (0,0)--(2,3.5)--(4,0)--cycle;

\node (1) [vertices] at (0,0) {};
\node (2) [vertices] at (2,3.5) {};
\node (3) [vertices] at (4,0) {};

\foreach \from/\to in {1/2,2/3,3/1}
	\draw [black!55] (\from)--(\to);

\draw [fill=black!05] (6,0)--(8,3.5)--(10,0)--cycle;

\node (4) [vertices] at (6,0) {};
\node (45) [vertices] at (7,1.75){};

\node (5) [vertices] at (8,3.5) {};
\node (56) [vertices] at (9,1.75){};

\node (6) [vertices] at (10,0) {};
\node (46) [vertices] at (8,0){};

\node (a) [vertices] at (7.3,.8){};
\node (b) [vertices] at (8.7,.8){};
\node (c) [vertices] at (8,1.9){};

\foreach \from/\to in {a/b,b/c,a/c,4/a,b/6,c/5,a/45,a/46,c/56,c/45,b/46,b/56}
	\draw [black!55] (\from)--(\to);

\end{tikzpicture}
\end{center}
\end{definition}
\begin{proposition}  Let $\Delta$ be a triangulation of a 2-sphere.  Then the Stanley-Reisner ideal of $\sdd(\Delta_G)$ is generated in degree 2.  Viewing this ideal as an edge ideal of a graph G, we have that $G^c$ is $C_4$-free.
\end{proposition}
\begin{proof}  As there are no induced 4-cycles inside an individual face $\sdd(\sigma)$, we may consider how facets $\sigma$, $\sigma'$ intersect after this subdivision.  As we assumed that $\Delta$ was a triangulation of a sphere, any two facets share at most one edge.  Along this edge, the only possible induced cycle is of length 6.  As every vertex must be in at least 3 facets $\sigma$, $\sigma'$, and $\sigma''$, we also note that every cycle obtained as the link of a vertex $v$ must be of length at least 6.  Performing all of these checks locally in the triangulation of $\Delta$, we see that the 1-skeleton of $\sdd(\Delta)$ must be $C_4$-free.\\
\\
We have that $\sdd(\Delta)$ is generated in degree 2 by noting that no boundaries of a 3-simplex can occur, so $\sdd(\Delta)$ must be clique closed, and hence, has a degree 2 generated Stanley-Reisner ideal.
\end{proof}
This provides a way of constructing an infinite family of $C-4$ free edge ideals from a large family of simplicial complexes.  A similar retriangulation exists for triangulations of the 3-sphere, to appear in \cite{Wh11}.  These infinite families provide large classes of graphs with $\Ind{G}=1$ and regularity 4 and 5 respectively.

\section{Conclusions and Future Work}
A better understanding of the possible shapes of Betti diagrams of edge ideals is desirable.  While a complete classification of the lower edges of Betti diagrams of edge ideals (or general monomial ideals) seems somewhat unrealistic, questions about their behavior have general applications to simplicial topology.  For example, sharp conditions for even a jump sequence of length 2 to exist translate into key information on necessary structure of triangulations of spheres -- an area of general combinatorial interests.  More nuanced questions about the Betti numbers and Stanley-Reisner complexes $\Delta_G$ of edge ideals include:
\begin{question}  Some sample open problems:
\begin{enumerate}
\item Can sharp conditions be given on possible jump sequences $[k;a_1,a_2,...,a_{k-1}]$?  Can sharp conditions even be given on jump sequences $[k;a_1,a_2]$, i.e. for graphs with $\reg(G)=4$?
\item Are the Betti diagrams of $I_G$ strand connected, i.e. if $\betti{i}{j}(I_G)$ and $\betti{i+k}{j+k}(I_G)$ are both nonzero, are $\betti{i+k'}{j+k'}(I_G)\neq 0$ for all $0\leq k'\leq k$? [This is known for the linear strand, but not even for the first nonlinear strand.]
\item Do there exist graphs with regularity higher than $\Ind{G}$ but lower than the co-chordal clutter size of $G$?
\item Do there exist graphs with $\Ind{G}=1$ and $\reg(I_G)=6$?
\item Do there exist graphs with $\Ind{G}=k$ and $\reg(I_G)\geq 4k+1$?
\end{enumerate}
\end{question}

\bibliographystyle{alpha}
\bibliography{JumpSequencePaper11-16-10}

\end{document}